\documentclass[12pt]{amsart}
\usepackage{epsf}
\usepackage{psfrag}
\usepackage{fullpage}
\usepackage{float}
\usepackage{mathrsfs}
\usepackage{amsfonts}
\usepackage[centertags]{amsmath}
\usepackage{amssymb}
\usepackage{amsthm}
\usepackage{graphicx}
\usepackage{float}
\usepackage[all]{xy}
\usepackage{tikz-cd}
\usepackage{comment}

\usepackage{youngtab}
\usepackage{tikz}
\usetikzlibrary[shapes]
\usepackage{multirow}
\usepackage{caption}
\usepackage{xparse}
\usepackage{lscape}
\usepackage{subcaption} 
\usepackage{enumerate,enumitem}
\usepackage[margin=1in]{geometry} 
\usepackage[pdf]{pstricks}
\usetikzlibrary{matrix,arrows,decorations.pathmorphing}

\usepackage{ytableau}

\usepackage{quiver}
\usepackage{adjustbox} 
\usepackage{tikz}
\usepackage{graphicx}
\usepackage[dvipsnames]{xcolor}
\usetikzlibrary{arrows.meta,positioning,backgrounds}
\usepackage[T1]{fontenc}
\usepackage{lmodern}
\usepackage{microtype}
\usepackage{parskip}
\pgfdeclarelayer{background}
\pgfsetlayers{background,main}
\usepackage{aliascnt}
\usepackage{hyperref}
\hypersetup{
    colorlinks=true,
    linkcolor=Maroon,
    filecolor=Magenta,      
    urlcolor=Maroon,
    citecolor=Maroon
}
\usepackage[capitalize,noabbrev,nameinlink]{cleveref}

\usetikzlibrary{calc}

\tikzset{
  edge/.style={
    line width=0.9pt,
    line cap=round,
    line join=round
  },
  medge/.style={
    line width=0.9pt,
    double,
    double distance=1.3pt,
    line cap=round
  },
  bdot/.style={
    circle,
    fill=black,
    inner sep=1.35pt
  },
  Bdot/.style={
    circle,
    fill=black,
    inner sep=1.7pt
  },
  wdot/.style={
    circle,
    draw=black,
    fill=white,
    line width=0.8pt,
    inner sep=1.8pt
  },
  lab/.style={
    font=\small
  }
}

\definecolor{srcred}{rgb}{0.78,0.05,0.07}
\newcommand{\BPl}[1]{\boldsymbol{P}_{#1}}
\newcommand{\RPl}[1]{{\color{srcred}\boldsymbol{P}_{#1}}}
\newcommand{\Rv}[2]{R^{#1}_{#2}}
\newcommand{\BRv}[2]{\boldsymbol{R}^{#1}_{#2}}
\newcommand{\RRv}[2]{{\color{srcred}\boldsymbol{R}^{#1}_{#2}}}
\newcommand{\Qv}[1]{Q_{#1}}
\newcommand{\BQv}[1]{\boldsymbol{Q}_{#1}}

\newcommand{\Pl}[1]{P_{#1}}

\newtheorem{theorem}{Theorem}[section]

\newaliascnt{proposition}{theorem}
\newtheorem{proposition}[proposition]{Proposition}
\aliascntresetthe{proposition}

\newaliascnt{lemma}{theorem}
\newtheorem{lemma}[lemma]{Lemma}
\aliascntresetthe{lemma}

\newaliascnt{corollary}{theorem}

\aliascntresetthe{corollary}

\newaliascnt{conjecture}{theorem}

\aliascntresetthe{conjecture}

\newaliascnt{question}{theorem}

\aliascntresetthe{question}

\theoremstyle{definition}
\newaliascnt{definition}{theorem}
\newtheorem{definition}[definition]{Definition}
\aliascntresetthe{definition}

\newaliascnt{convention}{theorem}

\aliascntresetthe{convention}

\newaliascnt{remark}{theorem}
\newtheorem{remark}[remark]{Remark}
\aliascntresetthe{remark}

\newaliascnt{example}{theorem}
\newtheorem{example}[example]{Example}
\aliascntresetthe{example}

\newaliascnt{warning}{theorem}

\aliascntresetthe{warning}

\theoremstyle{plain}
\newtheorem*{introtheorem}{Theorem}

\numberwithin{equation}{section}

\newcommand{\C}{{\mathbb C}}
\newcommand{\CC}{{\mathbb {C}}}

\newcommand{\Z}{{\mathbb {Z}}}

\newcommand{\PP}{{\mathbb {P}}}

\def\Acal{\mathcal{A}}
\def\A{\mathcal{A}}

\def\Xcal{\mathcal{X}}
\def\Fcal{\mathcal{F}}
\DeclareMathOperator{\arr}{arr}

\newcommand{\txx}{\widetilde{\mathbf{x}}}

\newcommand{\cyc}{{\operatorname{cyc}}}

\DeclareMathOperator{\refl}{refl}

\DeclareMathOperator{\SL}{SL}

\newcommand{\tpsi}{\widetilde{\psi}}

\DeclareMathOperator{\Gr}{Gr}

\begin{document}

\title{Grassmannian tree webs and cluster variables}

\author{Jian-Rong Li}
\address{Faculty of Mathematics, University of Seville, Calle Tarfia s/n, 41012 Seville, Spain}
\email{lijr07@gmail.com}

\author{Chenglu Wang}
\address{Department of Mathematics, Harvard University, Cambridge, MA, United States}
\email{cwang@math.harvard.edu}

\author{Lauren Williams}
\address{Department of Mathematics, Harvard University, Cambridge, MA, United States}
\email{williams@math.harvard.edu}

\date{}

\begin{abstract}
Homogeneous coordinate rings of Grassmannians are among the most fundamental instances of cluster algebras arising in ``nature.''  Despite their importance, however, the cluster variables of these  cluster algebras are not well-understood, except in special cases.  For the Grassmannian of $3$-planes in $n$-space, Fomin and Pylyavskyy conjectured \cite{FP16} that the cluster and frozen variables are exactly the indecomposable, non-elliptic web invariants which are \emph{arborizable}. They proved that in particular, if an $\mathrm{SL}_3$ tensor diagram is a planar tree, its web invariant is a cluster or frozen variable.  In this article, we give a new proof of their result for $\mathrm{Gr}(3,n)$ planar tree webs, and we generalize the result to $\mathrm{Gr}(4,n)$: if an $\mathrm{SL}_4$ tensor diagram for $\Gr(4,n)$ is a planar tree web, its web invariant is a cluster or frozen variable. Our proof uses some of the cluster quasihomomorphisms from \cite{even2023cluster} and \cite{plabictangle} (namely, upper promotion and spurion promotion), as well as a new \emph{6-leg pattern} map.
\end{abstract}

\maketitle
\setcounter{tocdepth}{1}
\tableofcontents

\section{Introduction}

When cluster algebras were discovered by Fomin and Zelevinsky around 
2000 \cite{FZ1},   Grassmannians \cite{scott} were among the first examples showing that cluster structures arise naturally as coordinate rings of 
algebraic varieties.  Nevertheless, we still do not have a concrete understanding of what are the cluster variables for Grassmannians, outside of the finite type cases: $\Gr(2,n)$, $\Gr(3,6)$, $\Gr(3,7)$, and $\Gr(3,8)$ \cite{scott}. 

In \cite{FP16}, Fomin and Pylyavskyy gave a beautiful conjectural characterization of the cluster variables for $\Gr(3,n)$ in terms of tensor diagrams for $\SL_3$ tensor invariants in the sense of Kuperberg \cite{Kuperberg_1996}. They conjectured that the cluster and frozen variables for $\Gr(3,n)$ are exactly the indecomposable, non-elliptic web invariants which are \emph{arborizable}, namely they admit two presentations: (1) a planar web, which may have cycles, and (2) a tree diagram, which may fail to be planar \cite[Conjecture 10.1]{FP16}. They proved that when both properties are satisfied simultaneously, the corresponding web invariant is a cluster variable \cite[Corollary 8.10]{FP16}. 

Since then, Fraser verified that in the finite mutation type cases ($\Gr(3,9)$ and $\Gr(4,8)$) \cite{Fraser2}, every cluster variable is an indecomposable arborizable web invariant.  More recently, Banaian, Catania, Gaetz, Moore, Musiker, and Wright
\cite{BCG} enumerated all indecomposable arborizable degree four web invariants for $\Gr(3,n)$, showing that their count agrees with the conjectural enumeration of degree four cluster variables obtained in \cite{clustering}; this provides evidence for the Fomin-Pylyavskyy conjecture for cluster variables of Pl\"ucker degree four.

In recent work on the amplituhedron \cite{even2023cluster}, Even-Zohar, Lakrec, Parisi, Sherman-Bennett, Tessler, and the third author of this paper used a  {cluster quasihomomorphism} called \emph{promotion} to prove the cluster adjacency conjecture for BCFW tiles.  Followup work \cite{plabictangle} developed the framework of \emph{plabic tangles}, which uses plabic graphs to define rational maps between Grassmannians which in nice cases are cluster quasihomomorphisms.

In this paper, we use the upper and spurion promotion maps from \cite{even2023cluster} and \cite{plabictangle}, as well as a new \emph{6-leg pattern} map, to prove the following result.

\begin{introtheorem}[\cref{thm:main1} and \cref{thm:main2}] \label{thm:main}
If an $\SL_3$ tensor diagram $D$ for $\Gr(3,n)$ is a tree web, then $[D]$ is a cluster or coefficient variable in $\C[\widehat{\Gr}(3,n)]$.
Similarly, if an $\SL_4$ tensor diagram $D$ for $\Gr(4,n)$ is a tree web, then $[D]$ is a cluster or coefficient variable in $\C[\widehat{\Gr}(4,n)]$.
\end{introtheorem}

This gives a new proof of \cite[Corollary 8.10]{FP16} in the case of $\Gr(3,n)$ and extends the result to the case of $\Gr(4,n)$.

The methods used here apply more broadly beyond $\Gr(3,n)$ and $\Gr(4,n)$ tree webs.  This will be explored in followup work.

\noindent{\bf Acknowledgements:~}  
C. W. would like to thank Thomas Brazelton and Ran Tessler for helpful conversations. J. L. was supported by the Beatriz Galindo Senior grant BG24/00114 from the Spanish Ministry of Science, Innovation and Universities. L.W. has been supported by the National Science Foundation under Award No. DMS-2152991 and DMS-2552947. Any opinions, findings, and conclusions or recommendations expressed in this material are
those of the author(s) and do not necessarily reflect the views of the National Science
Foundation. 

ChatGPT was used to help with typesetting hand-drawn figures and technical computations related to \cref{thm:6leg}. All resulting material was reviewed and verified by the authors. 

\section{Background on Grassmannians, tensor invariants, and webs}

\subsection{Grassmannians}
Let $\Gr(k,n)$ denote the
Grassmannian of 
$k$-dimensional subspaces
in~$\CC^n$. Any element of $\Gr(k,n)$  can be represented (non-uniquely) as the row span of a full rank $k \times n$ matrix $C$.
Given such a matrix and a $k$-element subset $J$ of $[n]=\{1,2,\dots,n\}$,
we let $P_J(C)$ denote the determinant of the $k \times k$ submatrix of $C$
located in columns $J$; this is called a \emph{Pl\"ucker coordinate}.
The Grassmannian $\Gr(k,n)$
can be embedded into projective space of dimension $\binom{n}{k}-1$
via the Pl\"ucker embedding $C \mapsto (P_J(C))_{J\in {[n] \choose k}}$.

Sometimes it will be convenient to use Pl\"ucker coordinates associated to 
\emph{sequences} $j_1,\dots,j_k$ of $[n]$, which are the determinants of the $k\times k$ matrices whose columns are (in order) the columns of $C$ specified by the sequence.  With this convention, the Pl\"ucker coordinates are alternating in the indices, e.g. 
$P_{1,2,3}=P_{2,3,1}=P_{3,1,2}=-P_{1,3,2}=-P_{2,1,3}=-P_{3,2,1}.$

When we have a $k$-element subset $J=\{j_1<j_2<\cdots<j_k\} \subset [n]$, the Pl\"ucker coordinate associated to that subset is understood to be the Pl\"ucker coordinate associated to the sequence of its elements written in increasing order, i.e. $P_J = P_{j_1, \ldots, j_k}$. 

Let $\widehat{\Gr}(k,n)$ denote the affine cone over ${\Gr}(k,n)$
in the Pl\"ucker embedding.
The homogeneous coordinate ring
 $\C[\widehat{\Gr}(k,n)]$
of~$\Gr(k,n)$
is generated by the Pl\"ucker coordinates~$P_J$,
where $J$ ranges over all elements of ${[n] \choose k}$.
The Pl\"ucker coordinates satisfy the \emph{Grassmann-Pl\"ucker relations}.

It was shown by Scott \cite{scott} that homogeneous coordinate rings of Grassmannians carry a cluster structure. In this paper, when we refer to cluster variables for the Grassmannian $\Gr(k,n)$, we always mean cluster variables in the homogeneous coordinate ring $\C[\widehat{\Gr}(k,n)]$.

\subsection{Invariant rings}
The homogeneous coordinate ring of the Grassmannian is an instance of an $\SL_k$-invariant polynomial ring; such rings have been studied since the early days of classical invariant theory, see for example \cite{WeylClassicalGroup} and references in \cite[Section 1]{FP16}. We follow \cite{FP16} to recall some key definitions and results. 

Fix $V$ a $k$-dimensional complex vector space, and denote $V^*$ as its dual space. In this manuscript, we focus on the cases of $k=3$ and $k=4$. We refer to elements in $V$ as \emph{vectors} and elements in $V^*$ as \emph{covectors}. A \emph{tensor} $T$ of \emph{type $(a,b)$}  is an element in $(V)^{\otimes a} \otimes (V^*)^{\otimes b}$, or alternatively, a multilinear map $T: (V^*)^a \times (V)^b \rightarrow \C$. 

The special linear group $\SL(V)$ acts on $V$ and $V^*$ naturally, so we can consider the $\SL_k$-invariant polynomial ring $$R_{a,b}(V) := \C[(V^*)^a \times (V)^b]^{\SL(V)}.$$ Note that $$R_{0,n}(V) = \C[(V)^n]^{\SL(V)} \cong \C [\widehat{\Gr}(k,n)]$$ for $n \geq k$ is the homogeneous coordinate ring of the Grassmannian $\Gr(k,n)$.

The First Fundamental Theorem of invariant theory implies that there are three basic $\SL(V)$-invariant tensors that build up a invariant polynomial ring $R_{a,b}(V)$: the volume tensor, the dual volume tensor, and the identity tensor. These three tensors can be expressed diagrammatically as white vertices, black vertices, and edges, and a general $\SL(V)$-invariant tensor can be expressed as a linear combination of \emph{tensor diagrams} built by them, drawn in a disk. When a tensor diagram is planar, we call it a \emph{web}. We refer readers to, for example, \cite{Temperley-Lieb, Kuperberg_1996, CKM, FP16}, for the theory of tensor diagram calculus.

We use $D$ to denote a tensor diagram (which is a graph) and $[D]$ to denote the corresponding $\SL_k$ tensor invariant. 

\subsection{$\SL_3$ tensor diagrams and webs} \label{SL_3 webs}

\begin{definition} (cf. \cite{Kuperberg_1996} and \cite[Definition 4.1]{FP16}) \label{SL3 web def}
An \emph{$\SL_3$ tensor diagram} is a finite bipartite graph $D$ drawn in a disk, where the vertex set is partitioned into boundary vertices and internal vertices, such that: 
\begin{enumerate}
    \item each internal vertex is trivalent and drawn inside the disk;
    \item each boundary vertex is drawn on the boundary of the disk;
    \item for each internal vertex, there is a fixed cyclic order on the edges incident to it.
\end{enumerate}

Note that drawing the bipartite graph $D$ in a disk induces a cyclic ordering of boundary vertices. We always label boundary vertices clockwise.

When the tensor diagram is planar, we call it an \emph{$\SL_3$ web}, and we call the drawing in the disk an \emph{embedding}. 

An internal white vertex represents a volume form, i.e. an element in $\Lambda^3 V^*$. An internal black vertex represents a dual volume form, i.e. an element in $\Lambda^3 V$. A boundary white vertex represents a covector going into the web. A boundary black vertex represents a vector going into the web.

The cyclic word over the alphabet $\{\text{black, white} \}$ of boundary vertices is called the \emph{signature} of the tensor diagram, and it represents which invariant ring $R_{a,b}$ the tensor invariant is in, see \cite[Section 1]{FP16}. In the case of $\Gr(3,n)$ tensor diagrams, all boundary vertices are black.
\end{definition}

\begin{definition} \label{tree web def}
    We call an $\SL_3$ web $D$ an \emph{$\SL_3$ tree web} if the following two conditions hold: \begin{enumerate}
        \item the underlying graph of $D$ has no cycles and is connected;
        \item all boundary vertices are univalent (having degree 1).
    \end{enumerate}
\end{definition}

\begin{remark} \label{boundary univalence rem}
    Fomin-Pylyavskyy also defined \emph{tree diagrams} in \cite[Section 10]{FP16}. We emphasize that they are different from tree webs: they are (not necessarily planar) diagrams that are trees after \emph{unclasping}. \emph{Unclasping} takes a boundary vertex $\beta$ of degree $t>1$ to $t$ copies of univalent boundary vertices $\beta_1, ..., \beta_t$, each is incident to one edge that was incident to $\beta$ before. 

    Although not explicitly stated in their paper, Fomin-Pylyavskyy's definition of a tree web requires it to have univalent boundary vertices (the same as in our Definition \ref{tree web def}).
    Their results say that planar tree webs are irreducible \cite[Lemma 12.2]{FP16} and are cluster variables \cite[Corollary 8.10]{FP16}. These would not be true if we don't require univalent boundary vertices: for example, the product of two Pl\"uckers $P_{123}P_{145}$ is not irreducible but is a planar tree with boundary vertex $v_1$ having degree 2.
\end{remark}

\begin{remark} \label{read off SL3 web rem}
Let $V$ be a 3-dimensional $\C$-vector space. When evaluating $\Gr(3,n)$ tree webs, we can choose an internal vertex as the root and orient all edges toward it to read off the tensor invariant in the following way \cite[Section 3.2]{RSV}: 

 Evaluating $v_1, v_2, v_3$ (ordered clockwise) at a white vertex is taking the determinant of the three vectors. Evaluating $c_1, c_2, c_3$ at a black vertex is taking $\det^* $ of the three covectors, where $\langle \det, \det^* \rangle = 1$ and $\langle \cdot, \cdot \rangle$ is the natural pairing between a vector space and its dual. 

A non-root internal white vertex takes in two vectors $v_1, v_2$ and outputs a covector $f$ (here $v_1, v_2$ and the covector are ordered clockwise), where $ f: V \rightarrow \C$ by $x \mapsto \det(v_1, v_2, x)$. Note that this is the cross product.
    
Similarly, a non-root internal black vertex is a map $$b: V^* \times V^* \rightarrow V \qquad (c_1, c_2) \mapsto {\det}^*(c_1, c_2, -).$$
\end{remark}

\begin{example}
Consider the following two $\Gr(3,7)$ webs.

  \begin{figure}[htbp]
    \centering
    \includegraphics[width=0.6\textwidth]{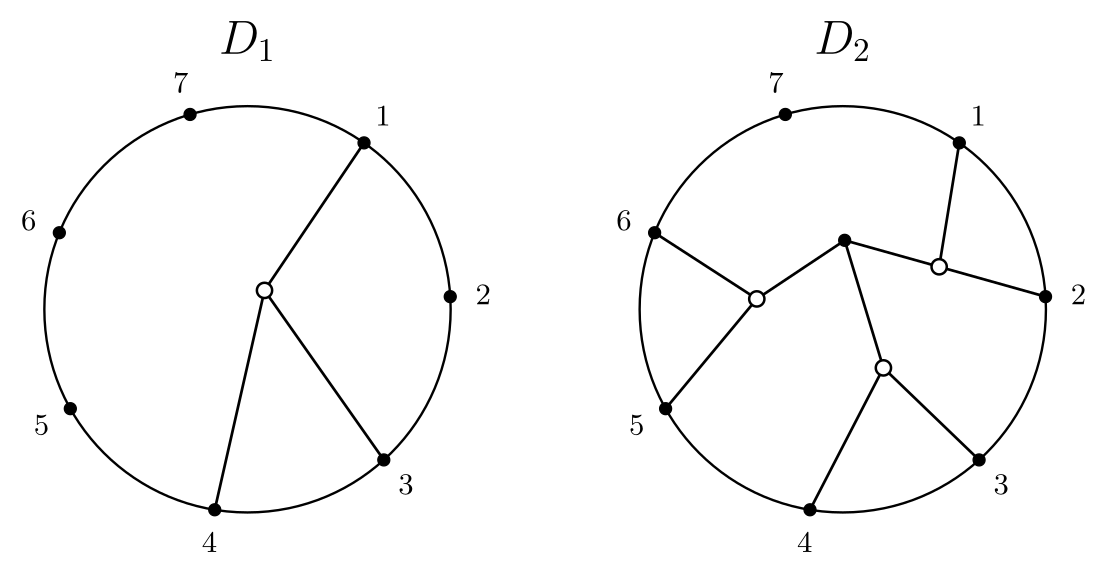}
    \caption{Two examples in $\C[\widehat{\Gr}(3,7)]$.}
    \label{two examples in Gr(3,7) figure}
\end{figure}

Evaluating at the white vertex in $D_1$, we have $$[D_1] = \det(v_1, v_3, v_4) = P_{134}.$$ We call $D_1$ a \emph{tripod}. In a $\Gr(3,n)$ tensor diagram, a tripod is a Pl\"ucker coordinate.

Evaluating at the internal black vertex in $D_2$, we have 
\begin{align*} 
[D_2] &=  {\det}^*(v_1 \times v_2, v_3 \times v_4, v_5 \times v_6) \\
&= \det(v_1, v_2, v_4)\det(v_3, v_5, v_6) - \det(v_1, v_2, v_3)\det(v_4, v_5, v_6)  \\ 
&= P_{124}P_{356} - P_{123}P_{456}.
\end{align*}
\end{example}

\subsection{$\SL_4$ tensor diagrams and webs}

Numerous authors have made progress toward higher $\SL_k$ webs; we refer readers to \cite[Section 1.3]{HourglassPlabic} for a detailed account. We will use the $\SL_k$ diagrammatic calculus that was first conjectured by Kim \cite{Kim2003} and Morrison \cite{MorrisonThesis} and later proved by Cautis--Kamnitzer--Morrison \cite{CKM}, with a different proof given by Fraser--Lam--Le \cite{FLL}.

We follow the conventions in \cite{FLL, Fraser2} to give a quick introduction to $\SL_4$ webs, mostly parallel to the case of $\SL_3$ in \cref{SL_3 webs}.

\begin{definition} (cf. \cite[Section 9.3]{Fraser2}) \label{Gr(4,n) web def}
    An \emph{$\SL_4$ tensor diagram} is a finite bipartite graph $D$ drawn in a disk, such that:
    \begin{enumerate}
        \item each edge is either of multiplicity 1 or 2, with the exception that edges connecting to a boundary vertex must be multiplicity 1;
        \item each internal vertex is of degree 4, counting edge multiplicity;
        \item no internal vertex is incident to exactly two edges both of multiplicity $2$.
    \end{enumerate}
Condition (3) is achieved by the two-valent vertex removal relation \cite[Section 6.1]{FLL}.

An \emph{$\SL_4$ web} is a planar $\SL_4$ tensor diagram. A \emph{$\Gr(4,n)$ tree web} is an $\SL_4$ web such that it has $n$ boundary vertices all black, is also a tree, and each boundary vertex has degree 1.
\end{definition}

\begin{remark} 
   Fix $V$ a 4-dimensional complex vector space. Similar to \cref{read off SL3 web rem}, we can choose an internal vertex as the root and orient all edges toward it to read off the tensor invariant in the following way \cite[Section 3.3]{RSV}: 

   A boundary black vertex is a vector. A boundary white vertex is a covector.
   
An internal black vertex is a contraction by $\epsilon^{abcd}$. That is, for a non-root black vertex, one of the following occurs:
    \begin{enumerate}
        \item It has three incoming single edges and one outgoing single edge, which represents a map $\mathcal{B}_1: \Lambda^3 V^* \rightarrow V$ by $\  A \wedge B \wedge C \mapsto Z$ where $(Z)^d  = \epsilon^{abcd} A_a B_b C_c$;
        \item It has two incoming single edges and one outgoing double edge, representing a map $\mathcal{B}_2: \Lambda^2 V^* \rightarrow \Lambda^2 V$ by $\  A \wedge B \mapsto Z$ where $(Z)^{cd}  = \frac{1}{2} \epsilon^{abcd} (A \wedge B)_{ab}$;
        \item It has two incoming edges, one single and one double, and there is one outgoing single edge, representing a map $\mathcal{B}_3: \Lambda^2 V^* \otimes V^* \rightarrow V$ by $\  (A \wedge B) \otimes C \mapsto Z$, where $(Z)^d  =  \frac{1}{2}\epsilon^{abcd} (A \wedge B)_{ab} C_c.$
\end{enumerate}
Here, we read edges incident to a vertex in clockwise order and take the outgoing edge $Z$ to be the last one in the ordering.

If a black vertex is the root, either:
 \begin{enumerate}
        \item It has four single edges (covectors) $A,B,C,D$ coming in, and evaluating the web is $\epsilon^{abcd}A_a B_b C_c D_d$; or
        \item It has one double edge $A \wedge B$ and two single edges $C,D$ coming in, and evaluating the web is $\frac{1}{2}\epsilon^{abcd} (A \wedge B)_{ab} C_c D_d$.
\end{enumerate}

At the root vertex, the evaluation is determined only up to a sign, depending on the chosen ordering of the edges.

An internal white vertex is a contraction by $\epsilon_{abcd}$, which is analogous to above but with all upper and lower indices interchanged. 
\end{remark}

\subsection{Tensor diagram calculus}

\begin{definition}
Let $D_1$ and $D_2$ be two tensor diagrams with the same boundary signature. Define the \emph{addition} $D_1 + D_2$ to be the formal sum of the two tensor diagrams. Define the \emph{multiplication} $D_1 \cdot D_2$ to be the superposition of the two tensor diagrams, identifying the boundary vertices. Multiplication may create crossings. 

\emph{Skein relations} are local equivalence relations between tensor diagrams that allow us to remove crossings, multiple edges, internal 4-cycles, loops, bigons, and boundary degeneracies, while preserving the corresponding tenor invariants. If $D \sim_{\text{skein}} \sum_i \lambda_i D_i$, then $[D] =\sum_i \lambda_i [D_i]$ as tensor invariants. See \cite{Kuperberg_1996} for the case of $\SL_3$, \cite{CKM} for the case of $\SL_4$, and, for example, \cite[Figure 5]{RSV} for a convenient list for both cases.
\end{definition}

\begin{remark}
    In \cite{Kuperberg_1996}, Kuperberg showed that \emph{non-elliptic} $\SL_3$ webs (obtained by skein relations) form a basis for the $\SL_3$ invariant ring. In contrast, the $\SL_4$ webs used in this paper only form a spanning set rather than a basis. We note, however, that web bases for higher $\SL_k$ are known \cite{Westbury_2011, FONTAINE20122792}, and, more recently, a rotation-invariant web basis for $\SL_4$ is constructed in \cite{HourglassPlabic}. We will not discuss these bases in this manuscript. 
\end{remark}

\section{Background on cluster algebras}

Cluster algebras were introduced by Fomin and Zelevinsky in \cite{FZ1}; see \cite{FWZ} for 
an introduction.
We give a quick definition of cluster algebras from quivers here.

\begin{definition}
[\emph{Quiver}]
	A \emph{quiver} $Q$ is an oriented graph given by a finite set of
	vertices $Q_0$, a finite set of arrows $Q_1$, and two maps
$s: Q_1 \to Q_0$ and $t: Q_1 \to Q_0$ taking an arrow to its source and target, respectively.
\end{definition}

\begin{definition}
[\emph{Quiver Mutation}]
Let $Q$ be a quiver without loops or $2$-cycles.
Let $k$ be a vertex of $Q$.
Following \cite{FZ1}, we define the
\emph{mutated quiver} $\mu_k(Q)$ as follows:
it has the same set of vertices as $Q$,
and its set of arrows is obtained by the following procedure:
\begin{enumerate}
\item for each subquiver $i \to k \to j$, add a new arrow $i \to j$;
\item reverse all allows with source or target $k$;
\item remove the arrows in a maximal set of pairwise
disjoint $2$-cycles.
\end{enumerate}
\end{definition}

It is not hard to check that mutation is an involution, that is,
$\mu_k^2(Q) = Q$ for each vertex $k$.

\begin{definition}
[\emph{Labeled seeds}]
\label{def:seed0}
Choose $s\geq r$ positive integers.
Let $\Fcal$ be an \emph{ambient field}
of rational functions
in $r$ independent
variables
over
$\CC(x_{r+1},\dots,x_s)$.
A \emph{labeled seed} in~$\Fcal$ is
a pair $(\txx, Q)$, where
\begin{itemize}
\item
$\txx = (x_1, \dots, x_s)$ forms a free generating
set for
$\Fcal$,
and
\item
$Q$ is a quiver on vertices
$1, 2, \dots,r, r+1, \dots, s$,
whose vertices $1,2, \dots, r$ are called
\emph{mutable}, and whose vertices $r+1,\dots, s$ are called \emph{frozen}.
\end{itemize}
We call~$\txx$  the (labeled)
\emph{extended cluster} of a labeled seed $(\txx, Q)$.
The variables $\{x_1,\dots,x_r\}$ are called \emph{cluster
variables}, and the variables $c=\{x_{r+1},\dots,x_s\}$ are called
	\emph{frozen} (or \emph{coefficient variables}).
We let $\PP$ denote the group of Laurent monomials in the frozen variables, which we call the \emph{frozen group}.
\end{definition}

\begin{definition}
[\emph{Seed mutations}]
\label{def:seed-mutation0}
Let $(\txx, Q)$ be a labeled seed in $\Fcal$,
and let $k \in \{1,\dots,r\}$.
The \emph{seed mutation} $\mu_k$ in direction~$k$ transforms
$(\txx, Q)$ into the labeled seed
$\mu_k(\txx,  Q)=(\txx', \mu_k(Q))$, where the cluster
$\txx'=(x'_1,\dots,x'_s)$ is defined as follows:
$x_j'=x_j$ for $j\neq k$,
whereas $x'_k \in \Fcal$ is determined
by the \emph{exchange relation}
\begin{equation}
\label{exchange relation0}
x'_k\ x_k =
 \ \prod_{\substack{\alpha\in Q_1 \\ s(\alpha)=k}} x_{t(\alpha)}
+ \ \prod_{\substack{\alpha\in Q_1 \\ t(\alpha)=k}} x_{s(\alpha)} \, .
\end{equation}
\end{definition}

Note that arrows between two frozen vertices of a quiver do not
affect seed mutation; therefore we often omit
arrows between two frozen vertices.

\begin{definition}
[\emph{Cluster algebra}]
\label{def:cluster-algebra0}
Let $(\txx, Q)$ be a labeled seed in $\Fcal$
and let $\Xcal$ denote the set of all cluster variables we can obtain by performing arbitrary sequences of mutations to the initial labeled seed.

Let $\CC[c^{\pm 1}]$ be the \emph{ground ring} consisting
of Laurent polynomials in the frozen variables. The
\emph{cluster algebra} $\Acal$ is the $\CC[c^{\pm 1}]$-subalgebra of the ambient field $\Fcal$
generated by all cluster variables, 
with coefficients which are Laurent polynomials
in the frozen variables: $\Acal = \CC[c^{\pm 1}] [\Xcal]$.
We denote $\Acal = \Acal(\txx,  Q)$.

We say that $\Acal$ has \emph{rank $r$} because each cluster contains
$r$ cluster variables. Cluster (or frozen) variables that belong to a common cluster are said to be \emph{compatible}.
\end{definition}

\begin{remark}\label{rmk:dif-ground-ring}
Another common convention is to choose the ring 
	$\CC[c]$ of polynomials in the frozen variables as the ground ring,
and define the cluster algebra 
as $\overline{\Acal} := \CC[c] [\Xcal]$.
\end{remark}

\subsection{The cluster structure on the Grassmannian}\label{sec:Grass}

The Grassmannian
has a cluster structure \cite{scott},
which we recall here, following the exposition of \cite{FW6}.
We also discuss several operations on the Grassmannian which are compatible
with the cluster structure.

Given a $k$-element subset $J=\{j_1,j_2, \dots, j_k\}\subset\{1,2,\dots, n\}$
and a positive integer $i$,
we define
\[
(J+i)\bmod n: = \{j_1+i, j_2+i, \dots, j_k+i\},
\]
where the sums are taken modulo~$n$.
We often write $J+i$ if the $n$ is clear
from context.  

The set of frozen variables for the cluster structure on the Grassmannian
consists of the $n$ \emph{cyclically consecutive} 
Pl\"ucker coordinates 
\begin{equation}\label{eq:Grfacets}
    P_{\{i,i+1,\ldots,i+k-1\}\bmod n}
    \quad\text{for } 1\leq i\leq n. 
\end{equation}

For example, the frozen variables  for $\Gr(3,6)$ are the Pl\"ucker coordinates
\[      
P_{123}, P_{234}, P_{345}, P_{456}, P_{156}, P_{126}.
\]

A particularly nice seed for the Grassmannian cluster structure is 
the \emph{rectangles seed} $\Sigma_{k,n}$.

\begin{definition}[Rectangles seed $\Sigma_{k,n}$] \label{def:rectangles}
We construct a quiver $Q_{k,n}$ whose vertices are labeled by the
rectangles contained in an $k \times (n-k)$ rectangle~$R$,
including the empty rectangle~$\varnothing$.
The frozen vertices of $Q_{k,n}$ are labeled by
the rectangles of sizes $k\times j$ (with $1\le j\le n-k$),
rectangles of sizes $i\times (n-k)$ (with $1\le i\le k$), and the empty rectangle.
The arrows from an $i\times j$ rectangle go to rectangles of sizes $i\times (j+1)$,
$(i+1)\times j$, and $(i-1)\times (j-1)$ (assuming those rectangles have nonzero
dimensions, fit inside~$R$, and the arrow does not connect two frozen vertices).
There is also an arrow from the frozen vertex labeled by~$\varnothing$
to the vertex labeled by the $1\times 1$ rectangle.
See  \Cref{fig:G36-Le-quiver}, left.

We map each rectangle $r$ contained in the $k \times (n-k)$ rectangle
$R$ to an $k$-element subset of
$\{1,2,\dots, n\}$ (representing a Pl\"ucker coordinate), as follows.
We justify $r$
so that its upper left corner coincides with the upper left corner
of $R$.  There is a path of length $n$ from the northeast corner of
$R$ to the southwest corner of $R$ which cuts out the
smaller rectangle $r$;  we label the steps of this path from $1$ to~$n$.
We then map $r$ to the set of labels $J(r)$ of the vertical steps on this path.
This construction allows us to assign to each vertex of the quiver $Q_{k,n}$
a particular Pl\"ucker coordinate.
We set
\begin{equation*}
\txx^{k,n} =
	\{P_{J(r)} \ \vert \ \text{ $r$ is a rectangle contained in an
$k \times (n-k)$ rectangle}\} ,
\end{equation*}
and then define the \emph{rectangles seed} $\Sigma_{k,n}=(\txx^{k,n}, Q_{k,n})$.
\end{definition}
The case $k=3$ and $n=6$ is shown in \cref{fig:G36-Le-quiver}.

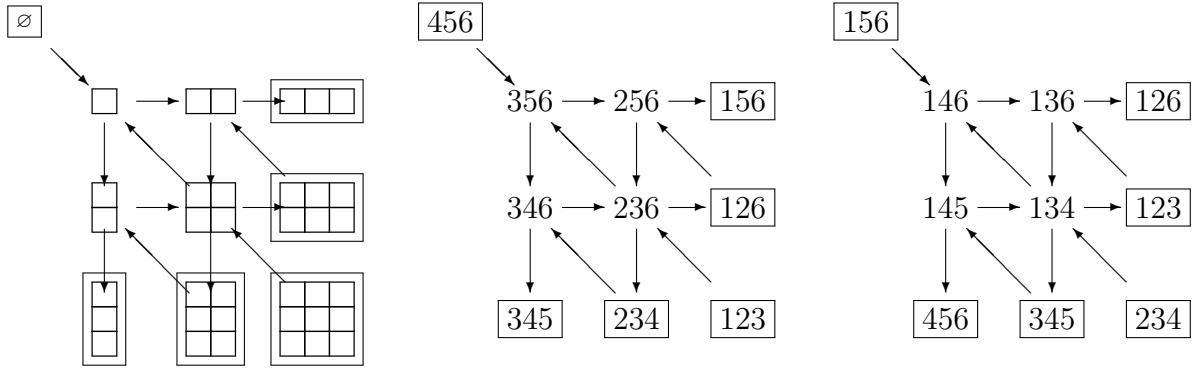
\begin{figure}[t] 
\vspace{1em}
\begin{center}
\setlength{\unitlength}{2pt}
\hspace{-1.5cm}
\ytableausetup{boxsize=.75em}
\begin{picture}(60,62)(0,-6)
\put(20,20){\makebox(0,0){${\ydiagram{1,1}}$}}
\put(20,40){\makebox(0,0){${\ydiagram{1}}$}}
\put(40,20){\makebox(0,0){${\ydiagram{2,2}}$}}
\put(40,40){\makebox(0,0){${\ydiagram{2}}$}}
\put(60,20){\makebox(0,0){$\boxed{\ydiagram{3,3}}$}}
\put(60,40){\makebox(0,0){$\boxed{\ydiagram{3}}$}}
\put(20,-1){\makebox(0,0){$\boxed{\ydiagram{1,1,1}}$}}
\put(40,-1){\makebox(0,0){$\boxed{\ydiagram{2,2,2}}$}}
\put(60,-1){\makebox(0,0){$\boxed{\ydiagram{3,3,3}}$}}
\put(5,55){\makebox(0,0){$\boxed{\scriptstyle\varnothing}$}}

\put(40,36){\vector(0,-1){12}}
\put(20,36){\vector(0,-1){12}}
\put(40,16){\vector(0,-1){12}}
\put(20,16){\vector(0,-1){12}}

\put(26,20){\vector(1,0){8}}
\put(46,40){\vector(1,0){8}}

\put(26,40){\vector(1,0){8}}

\put(46,20){\vector(1,0){8}}

\put(10,50){\vector(1,-1){7}}

\put(36, 4){\vector(-1,1){12}}
\put(54, 6){\vector(-1,1){10}}
\put(36, 24){\vector(-1,1){12}}
\put(54, 26){\vector(-1,1){10}}
\end{picture}
\ytableausetup{boxsize=normal}
\hspace{1cm}
\begin{picture}(60,62)(0,-6)
        \put(20,20){\makebox(0,0){$346$}}
        \put(20,40){\makebox(0,0){$356$}}
        \put(40,20){\makebox(0,0){${236}$}}
        \put(40,40){\makebox(0,0){${256}$}}
        \put(60,20){\makebox(0,0){$\boxed{126}$}}
        \put(60,40){\makebox(0,0){$\boxed{156}$}}
        \put(20,-1){\makebox(0,0){$\boxed{345}$}}
        \put(40,-1){\makebox(0,0){$\boxed{234}$}}
        \put(60,-1){\makebox(0,0){$\boxed{123}$}}

        \put(5,55){\makebox(0,0){$\boxed{456}$}}

        \put(40,36){\vector(0,-1){12}}
        \put(20,36){\vector(0,-1){12}}
        \put(40,16){\vector(0,-1){12}}
        \put(20,16){\vector(0,-1){12}}

        \put(26,20){\vector(1,0){8}}
        \put(46,40){\vector(1,0){7}}

        \put(26,40){\vector(1,0){8}}

        \put(46,20){\vector(1,0){7}}

        \put(10,50){\vector(1,-1){7}}

        \put(36, 4){\vector(-1,1){12}}
        \put(54, 6){\vector(-1,1){10}}

        \put(36, 24){\vector(-1,1){12}}
        \put(54, 26){\vector(-1,1){10}}
\end{picture}
\hspace{1cm}
\begin{picture}(60,62)(0,-6)
        \put(20,20){\makebox(0,0){$145$}}
        \put(20,40){\makebox(0,0){$146$}}
        \put(40,20){\makebox(0,0){${134}$}}
        \put(40,40){\makebox(0,0){${136}$}}
        \put(60,20){\makebox(0,0){$\boxed{123}$}}
        \put(60,40){\makebox(0,0){$\boxed{126}$}}
        \put(20,-1){\makebox(0,0){$\boxed{456}$}}
        \put(40,-1){\makebox(0,0){$\boxed{345}$}}
        \put(60,-1){\makebox(0,0){$\boxed{234}$}}

        \put(5,55){\makebox(0,0){$\boxed{156}$}}

        \put(40,36){\vector(0,-1){12}}
        \put(20,36){\vector(0,-1){12}}
        \put(40,16){\vector(0,-1){12}}
        \put(20,16){\vector(0,-1){12}}

        \put(26,20){\vector(1,0){8}}
        \put(46,40){\vector(1,0){7}}

        \put(26,40){\vector(1,0){8}}

        \put(46,20){\vector(1,0){7}}

        \put(10,50){\vector(1,-1){7}}

        \put(36, 4){\vector(-1,1){12}}
        \put(54, 6){\vector(-1,1){10}}

        \put(36, 24){\vector(-1,1){12}}
        \put(54, 26){\vector(-1,1){10}}
\end{picture}
\end{center}
\vspace{1em}
\caption{
Left: the quiver $Q_{3,6}$ with vertices labeled by rectangles contained in a $3 \times 3$ rectangle.
Middle: the rectangles seed $\Sigma_{3,6}$, where we identify $3$-element subsets of $[6]$ with Pl\"ucker coordinates. Frozen variables are boxed. Right: the cyclically shifted rectangles seed $\Sigma_{3,6}^1=\cyc(\Sigma_{3,6})$.
}
\label{fig:G36-Le-quiver}
\end{figure}

The rectangles seed gives rise to a cluster
structure on the (complex) Grassmannian.

\begin{theorem}[\cite{scott}]\label{thm:scott}
The coordinate ring $\C[\widehat{\Gr}(k,n)]$ of the affine cone over
the Grassmannian equals the cluster algebra $\overline{\Acal}(\Sigma_{k,n})$ (see \Cref{rmk:dif-ground-ring}). 
Alternatively,
if we let 
${\Gr}(k,n)^\circ$ denote the open subset of the Grassmannian where the frozen variables
don't vanish, then 
the coordinate ring $\C[\widehat{\Gr}(k,n)^\circ]$ is the cluster algebra
 $\A(\Sigma_{k,n})$.
\end{theorem}

Whenever we refer to ``the cluster structure for'' or ``a seed for'' $\widehat{\Gr}(k,n), \widehat{\Gr}(k,n)^{\circ}$ or their coordinate rings, we are refering to the cluster algebra, respectively a seed in the cluster algebra, in \Cref{thm:scott}. Note in this context, we are discussing the complex Grassmannian.

\subsection{Operations compatible with the cluster structure on the Grassmannian}\label{sec:operation}

\begin{proposition}\label{cor:cycrefl}\cite[Corollary 4.18]{even2023cluster}
Let  $\cyc^*: \C[\widehat{\Gr}(k,n)] \to \C[\widehat{\Gr}(k,n)]$ be the map which sends the Pl\"ucker coordinate labeled $J$ to $(J+1) \bmod n$,
and let $\refl^*: \C[\widehat{\Gr}(k,n)] \to \C[\widehat{\Gr}(k,n)]$ be the map which sends the Pl\"ucker coordinate 
 labeled
  $J=\{j_1 < \dots < j_k\}$ in $Q_{k,n}$ to the Pl\"ucker coordinate labeled 
  $\{n-j_k+1,\dots ,n-j_2+1, n-j_1+1\}.$ 
       
        The maps $\cyc^*, \refl^*: \C[\widehat{\Gr}(k,n)] \to \C[\widehat{\Gr}(k,n)]$ take cluster variables to cluster variables and preserve compatibility and exchange relations.
\end{proposition}

 \cref{lem:add-marker-embed-gr} gives an inclusion of Grassmannian cluster structures
 that will be useful in our proofs.
  \begin{lemma}\label[lemma]{lem:add-marker-embed-gr}\cite[Lemma 4.20]{even2023cluster}
          Let $T \subset [n]$. Choose $t \in T$ and set $S := T \setminus \{t\}$. Consider the natural inclusion $\iota: \C[\widehat{\Gr}(k, S)] \to \C[\widehat{\Gr}(k, T)]$ given by $P_I \mapsto P_I$. Then $\iota$ maps cluster variables to cluster variables and preserves compatibility and exchange relations.
         \end{lemma}

\subsection{Quasi-homomorphisms of cluster algebras}
	
In this section we define cluster quasi-homomorphisms, following \cite{Fraser} and \cite{Fraser2}.

	\begin{definition}[Exchange ratios] Given a cluster seed $\Sigma=((x_1,\dots,x_s), Q)$ 
		for a cluster algebra of rank $r\leq s$,
		and a mutable
		variable $x_i$ (so that $1 \leq i \leq r$), 
		the \emph{exchange ratio} of
		$x_i$ (with respect to $\Sigma$) is 
		\begin{equation}
			\hat{y}_{\Sigma}(x_i) = 
			\frac{\prod_{j: i\to j} x_j^{\# \arr(i\to j)}}
			{\prod_{j: j\to i} x_j^{\# \arr(j\to i)}},
		\end{equation}
		where $\arr(i\to j)$ denotes the number of arrows from 
		$i$ to $j$ in the quiver $Q$.
	\end{definition}

Given a cluster algebra $\A$, we let $\PP$ denote its frozen group, that is the group of Laurent monomials
in the frozen variables.  For elements $x,y\in \A$, we say that 
\emph{$x$ is proportional to $y$}, writing $x \propto y$, if $x=My$ for some Laurent monomial
$M\in \PP$. We then refer to $M$ as a \emph{frozen factor}.

\begin{definition}[Cluster quasi-homomorphism]\label{def:quasi}\cite[Definition 3.1 and Proposition 3.2]{Fraser}
Let $\A$ and $\overline{\A}$ be two cluster algebras of the same rank $r$,
and with respective frozen groups $\PP$ and $\overline{\PP}$.
	Then an algebra homomorphism $f:\A \to \overline{\A}$ that satisfies 
	$f(\PP)\subseteq \overline{\PP}$ is called a \emph{(cluster) quasi-homomorphism} from $\A$ to $\overline{\A}$
	if there are seeds
$\Sigma=((x_1,\dots,x_s),Q)$ and $\overline{\Sigma}=((\bar{x}_1,\dots,\bar{x}_{\bar{s}}), \bar{Q})$ 
	for $\A$ and $\overline{\A}$, such that 
	\begin{enumerate}
		\item $f(x_i) \propto \bar{x}_i$ for $1\leq i \leq r$
		\item $f(\hat{y}_{\Sigma}(x_i)) = \hat{y}_{\overline{\Sigma}}(\bar{x}_i)$ for $1\leq i \leq r$.
		\item the map $i \mapsto \bar{i}$ of mutable nodes in $Q$ and $\bar{Q}$ extends to 
                    an isomorphism of the corresponding induced subquivers. 
	\end{enumerate}
	If conditions (1), (2), and (3) hold, we say that 
	$\Sigma$ is \emph{similar} to $\overline{\Sigma}$, and 
	we write $f(\Sigma) \propto \overline{\Sigma}$.

Note that conditions (1) and (2) above imply condition (3), 
as they allow one to read off adjacencies of mutable nodes;
however, we choose to include condition (3) in the definition 
since it is readily 
	checkable and hence useful when looking $\Sigma$ and $\overline{\Sigma}$.

	\end{definition}

\begin{proposition}
\cite[Proposition 3.2]{Fraser} \label{prop:similar}
If a  seed
$\Sigma$ is similar to the seed $\overline{\Sigma}$, then 
each seed of the seed pattern containing $\Sigma$ is similar to the 
	corresponding seed of the seed pattern containing $\overline{\Sigma}$.
\end{proposition}

\section{Every $\Gr(3,n)$ tree web represents a cluster variable}

The goal of this section is to give a new proof of the following theorem, which was previously proved in \cite[Corollary 8.10]{FP16}.

\begin{introtheorem}[\cref{thm:main1}]
If an $\SL_3$ tensor diagram $D$ for $\Gr(3,n)$ is a tree web in the sense of \cref{tree web def}, then 
$[D]$ is a cluster or coefficient variable in $\Gr(3,n)$.
\end{introtheorem} 

The heuristic of the proof is the following: We first show that every $\Gr(3,n)$ tree web is inductively generated by adding ``4-leg patterns'' to a tripod $P_{abc}$. We then recall the unary star promotion for $\Gr(3,n)$, and show that this cluster quasihomomorphism, up to a frozen factor, corresponds to ``adding a 4-leg pattern'' on the tree web. Hence every tree web lies in the iterated image of the promotion map and is a cluster or coefficient variable.

\subsection{Every $\Gr(3,n)$ tree web can be constructed from adding 4-leg patterns} 

\begin{remark} \label{spreading index remark}
     Recall from \cref{lem:add-marker-embed-gr} that the spreading index map is a cluster quasihomomorphism and almost tautological. Throughout this manuscript, unless otherwise specified, we consider both graph-theoretical operations and cluster quasihomomorphisms up to this spreading index map. For example, a tripod $P_{123}$ and a tripod $P_{abc}$ for any $a < b < c$ are considered to be the same up to this spreading index operation.
     
\end{remark}

\begin{definition} \label{4-leg pattern def}
    Let $D$ be a $\Gr(3,n)$ web diagram. Define a \emph{4-leg pattern} $\Lambda_{n,1,2,3}$ to be a bipartite subgraph of $D$ with vertices $\{v_{n}, v_1, v_2, v_3, W_1, W_2, B, U \}$, highlighted in red in the left hand side of \cref{4-leg pattern labeling}. The vertices $v_{n},v_1,v_2, v_3$ are boundary leaves, and $W_1, W_2, B, U$ are internal. The vertex $U$ is further connected to the rest of the graph, and the two unnamed black vertices may be leaves or internal vertices. 
    
    Since $\SL_3$ webs are invariant under dihedral group action on the boundary indices, for any 4-leg pattern whose legs are $v_{a-1}, v_{a}, v_{a+1}, v_{a+2}$ (we \emph{do} require legs to be cyclically consecutive), we may rotate the leg and make it become $\Lambda_{n,1,2,3}$. Thus, without loss of generality, when a tree web has a 4-leg pattern, we assume it to be $\Lambda_{n,1,2,3}$ and write it simply as $\Lambda$.
\end{definition}

\begin{definition} \label{adding and erasing 4-leg pattern def}
Let $T$ be a $\Gr(3,n)$ tree web that has a 4-leg pattern labeled as in the left hand side of \cref{4-leg pattern labeling}. Define the \emph{erasing a 4-leg pattern map} $\rho$ (a set-theoretical map on underlying vertex set and edge set) to be the map that erases this 4-leg pattern and replaces it with one edge $(U,v_n)$. Note that erasing a 4-leg pattern decreases the number of (connected) boundary vertices by 3, in particular making $v_1, v_2, v_3$ having degree 0. 

Let $T'$ be a $\Gr(3,n)$ tree web without any assumptions. We may equivalently think of it as a $\Gr(3,[4,n+3])$ tree web having $v_1, v_2, v_3$ as boundary vertices of degree 0. In $T'$, there is always a boundary leaf connected to its parent, and without loss of generality let this boundary leaf be $v_n$ and its parent be $U$. Define the \emph{adding a 4-leg pattern map} $\alpha$ that replaces this edge $(U,v_n)$ by the 4-leg pattern with legs $v_{n}, v_1, v_2, v_3$. 

It is easy to check that both operations take a tree web to a tree web, and they are inverses of each other in the sense that $\rho(\alpha(T')) = T'$ for any tree web $T'$, and $\alpha(\rho(T)) = T$ for any tree $T$ that has a 4-leg pattern (up to a dihedral group action). In fact, we will prove in \cref{inductive generation prop} that any tree having at least 6 boundary vertices has a 4-leg pattern.
\end{definition}

\begin{figure}[htbp]
    \centering
    \includegraphics[width=0.7\textwidth]{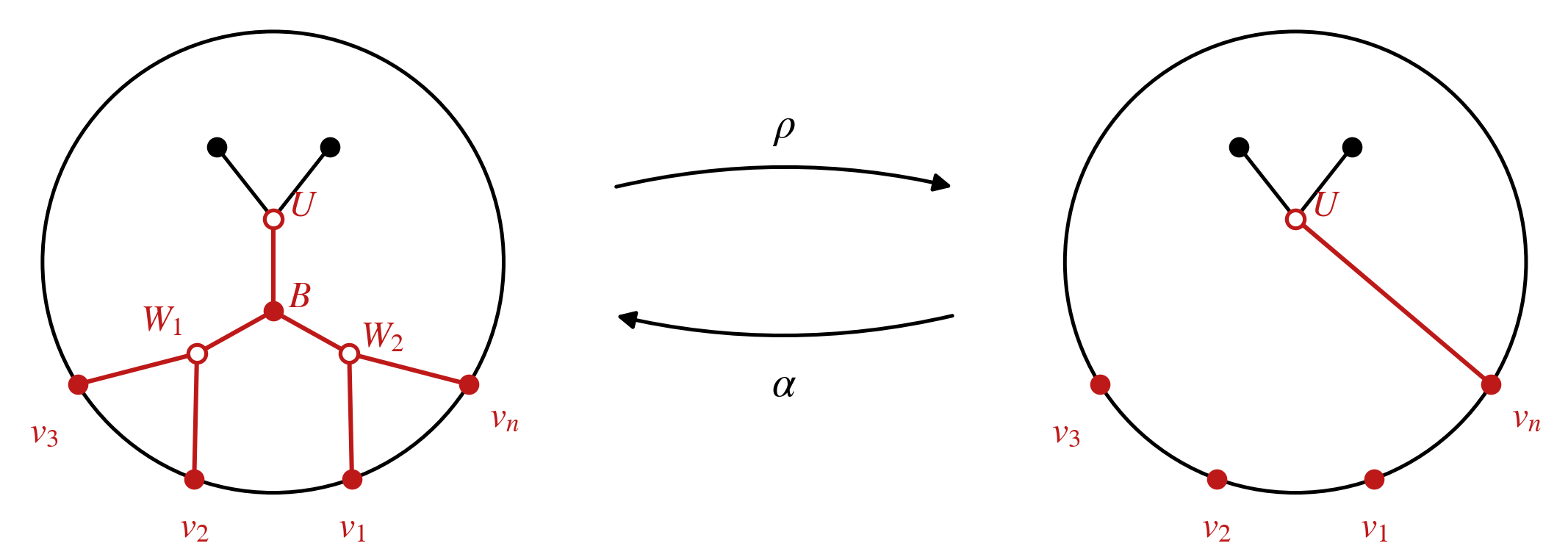}
    \caption{Labeling of a 4-leg pattern, the ``erasing a 4-leg pattern map'' $\rho$, and the ``adding a 4-leg pattern map'' $\alpha$.}
    \label{4-leg pattern labeling}
\end{figure}

The following lemma can be easily proved by induction. 
\begin{lemma} \label{3m boundary vertices}
    Every $\Gr(3,n)$ tree web has $3m$ boundary vertices 
    of degree $1$ (i.e. connected to the tree), for some $m \in \Z^+$, and $n = 3m$. 
\end{lemma}
\begin{proposition} \label{inductive generation prop}
Every $\Gr(3,n)$ planar tree web can be inductively generated by adding 4-leg patterns to a tripod $P_{abc}$.
\end{proposition}

\begin{proof}
    Let $T = (V,E)$ be a $\Gr(3,n)$ tree web. We prove the lemma by induction on the number of boundary vertices of $T$. By \cref{3m boundary vertices}, $T$ has $n = 3m$ boundary vertices. 
 Without loss of generality, by \cref{lem:add-marker-embed-gr}, 
 we can assume that all boundary vertices are connected to the rest of the web.
    
    When $m = 1$, $T$ is a tripod $P_{abc}$, which proves the base case.

    Now consider a tree web $T$ with $3m$ boundary vertices 
    for $m>1$.     Pick any white vertex (which must be internal) 
    and designate it as the root $R$ of the tree. Let $B$ be an internal black vertex which maximizes the distance (among all internal black vertices) to $R$.  Then $B$ has one white parent vertex $U$ and two white children vertices $W_1$ and $W_2$.
    
    Each of $W_1$ and $W_2$ have two black children vertices which are further from $R$ than $B$ is, so those four black children vertices must be boundary vertices.  Since $T$ is planar and all boundary vertices have degree $1$, these four black children vertices must be consecutive boundary vertices as in the figure below.  This is precisely a $4$-leg pattern.

\[
\begin{tikzcd}[column sep=0.7em, row sep=0.7em]
	&&& {\circ U} & \\
	&& {\bullet B} \\
	& {\circ W_1} && {\circ W_2} \\
	{\bullet C} & {\bullet C'} && {\bullet C''} & {\bullet C'''}
	\arrow[no head, from=2-3, to=1-4]
	\arrow[no head, from=2-3, to=3-2]
	\arrow[no head, from=2-3, to=3-4]
	\arrow[no head, from=3-2, to=4-1]
	\arrow[no head, from=3-2, to=4-2]
	\arrow[no head, from=3-4, to=4-4]
	\arrow[no head, from=3-4, to=4-5]
\end{tikzcd}
\]

Since we have found a 4-leg pattern $\Lambda$ in the tree $T$, we can apply the erasing map $\rho$, and we obtain $T' = \rho(T)$, which is a $\Gr(3,n)$ tree web on $3m-3$ boundary vertices. By induction, $T'$ can be obtained by adding 4-leg patterns to a tripod, so we are done. 
\end{proof}

\subsection{Unary star promotion}

\begin{definition} \label{Gr(3,n) promotion map}
    Consider the  map $\Gr(3,n) \to \Gr(3, [4,n])$ which sends the $3 \times n$ matrix $Z=(Z_1\mid \ldots \mid Z_n)$ to the 
$3 \times (n-3)$ matrix $Z' = (Z'_4\mid \ldots \mid Z'_n)$, where $Z'_i=Z_i$ for $i \ne n$ and $Z'_n = \frac{P_{2,3,n}}{P_{1,2,3}} Z_{1} - Z_n$. 
This map induces a homomorphism $\psi: \CC[\widehat{\Gr}(3,[4,n])] \to \C[\widehat{\Gr}(3,n)]$ called the \emph{unary star promotion} which maps Pl\"ucker coordinates as follows:
\begin{itemize}
    \item if $n \in I$, then  $$P_I' \mapsto \frac{P_{23n}}{P_{123}} P_{I - \{ n \} \cup \{1 \} } - P_I.$$
    \item if $n \notin I$, then $P_I' \mapsto P_I$.
\end{itemize}
Above, $I$ is a subset of $[4,n]$ with size $3$, and $P'_{I}$ is the minor of $Z' = (Z'_4\mid \ldots \mid Z'_n)$ with columns $Z'_i$, $i \in I$. 
\end{definition}

\begin{theorem}\cite[Theorem 8.1]{plabictangle}
Unary star promotion is a cluster quasihomomorphism from $\C[\widehat{\Gr}(3,[4,n])] \to \C[\widehat{\Gr}(3, n)]$, where the cluster structure of $ \C[\widehat{\Gr}(3, [4,n])]$ is the natural restriction of the one for $\C[\widehat{\Gr}(3, n)]$.
\end{theorem}

\begin{example} \label{4-leg and promotion correspondence ex}
   Let $D_1$ be the tree web correspond to the cluster variable $P_{457}P_{689}-P_{456}P_{789}$. Applying $\psi$ gives 
\begin{align*}
\psi(P_{457}P_{689}-P_{456}P_{789}) & = P_{457} (\frac {P_{239}}{P_{123}} P_{168}- P_{689}) - P_{456} (\frac {P_{239}}{P_{123}}P_{178}- P_{789}) \\
 &= \frac{1}{P_{123}} (P_{457}P_{239}P_{168} - P_{123} P_{457} P_{689} - P_{456}P_{239}P_{178} + P_{123} P_{456}P_{789}).
\end{align*}
Since $P_{123}$ is a frozen variable, we can disregard the coefficient $\frac{1}{P_{123}}$, and the resulting cluster variable $(P_{457}P_{239}P_{168} - P_{123} P_{457} P_{689} - P_{456}P_{239}P_{178} + P_{123} P_{456}P_{789})$ corresponds to the web $D$. 

Graphically, this example is shown in \cref{skein rel and tpsi example figure}. Observe that applying $P_{123} \cdot \psi$ to $[D']$ precisely corresponds to adding a 4-leg pattern to $D'$. We will prove that this correspondence is true for any tree web (assuming $D'$ having degree 0 boundary vertices $v_1, v_2, v_3$) in \cref{adding 4-leg is applying psi}.

  \begin{figure}[htbp]
    \centering
    \includegraphics[width=0.75\textwidth]{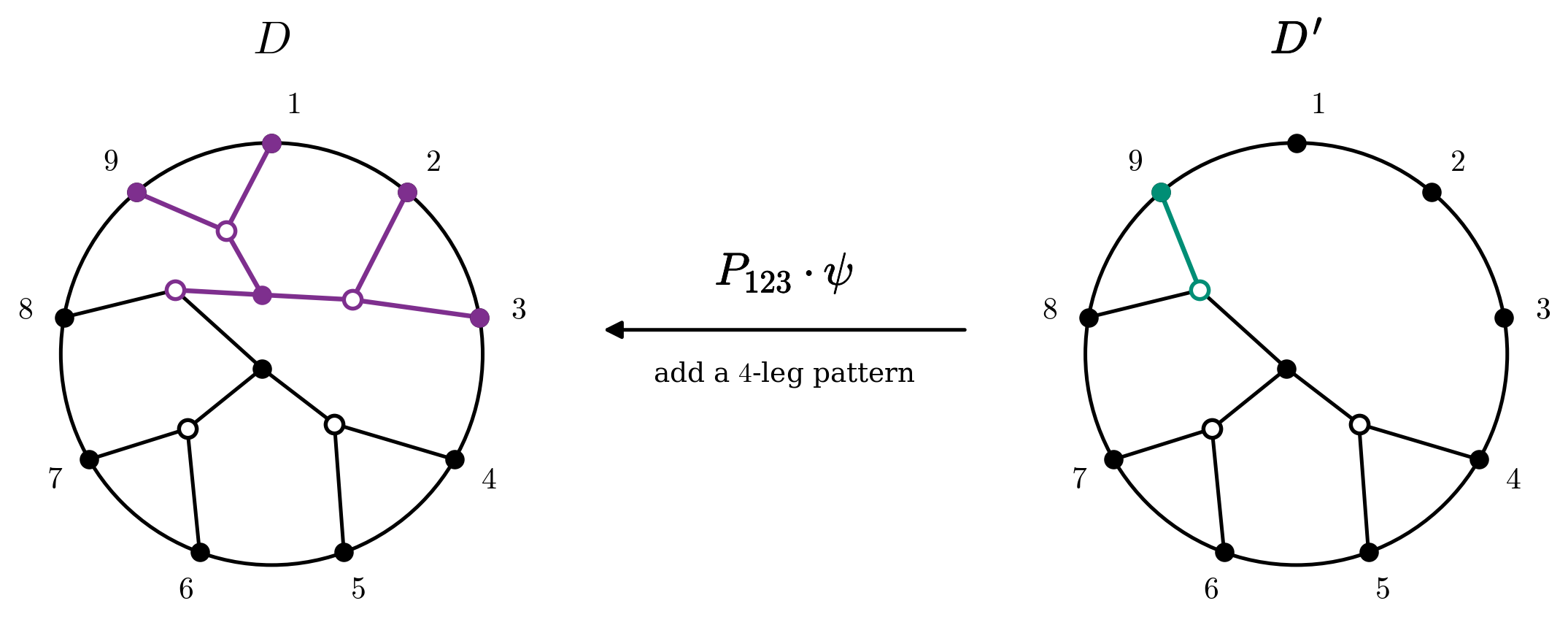}
    \caption{An illustration of \cref{4-leg and promotion correspondence ex}.}
    \label{skein rel and tpsi example figure}
\end{figure}

\end{example}

\subsection{The adding 4-leg pattern operation corresponds to the promotion map}\label{sec:local}
 
Let $T$ be a $\Gr(3,n)$ tree web. We consider it as a $\Gr(3, [4,n+3])$ tree web whose boundary vertices $v_1,v_2,v_3$ have degree 0. We now show that adding a 4-leg pattern to $T$ corresponds to applying $P_{123} \cdot \psi$ to $[T]$, where $\psi$ is the unary star promotion defined in Definition \ref{Gr(3,n) promotion map}.

The following \emph{vector quadruple-product identity} will be useful to us.

\begin{lemma} \label{multi linear identity}
    Let $V$ be a 3-dimensional $\C$-vector space. Recall \cref{read off SL3 web rem} that in an $\SL_3$ web, an internal black vertex takes in two covectors and outputs a vector by $$b: V^* \times V^* \rightarrow V, \qquad (c_1, c_2) \mapsto {\det}^*(c_1, c_2, -). $$ An internal white vertex takes in two vectors and outputs a covector as the cross product $$\times: V \times V \rightarrow V^*, \qquad (v,w) \mapsto \det(v,w, -).$$ Then $$b(A \times B, C \times D) = \det(C,D,A)B - \det (B,C,D)A.$$
\end{lemma}

\begin{theorem} \label{adding 4-leg is applying psi}
    Let $D'$ be a $\Gr(3, [4,n+3])$ tree web, and write $v_1, v_2, v_3$ as boundary vertices of degree 0. Let $\alpha(D') = D$ be the tree web obtained by adding a 4-leg pattern from $D'$. Then $P_{123} \cdot \psi([D']) = [D]$.
\begin{figure}[htbp]
    \centering
    \includegraphics[width=0.8\textwidth]{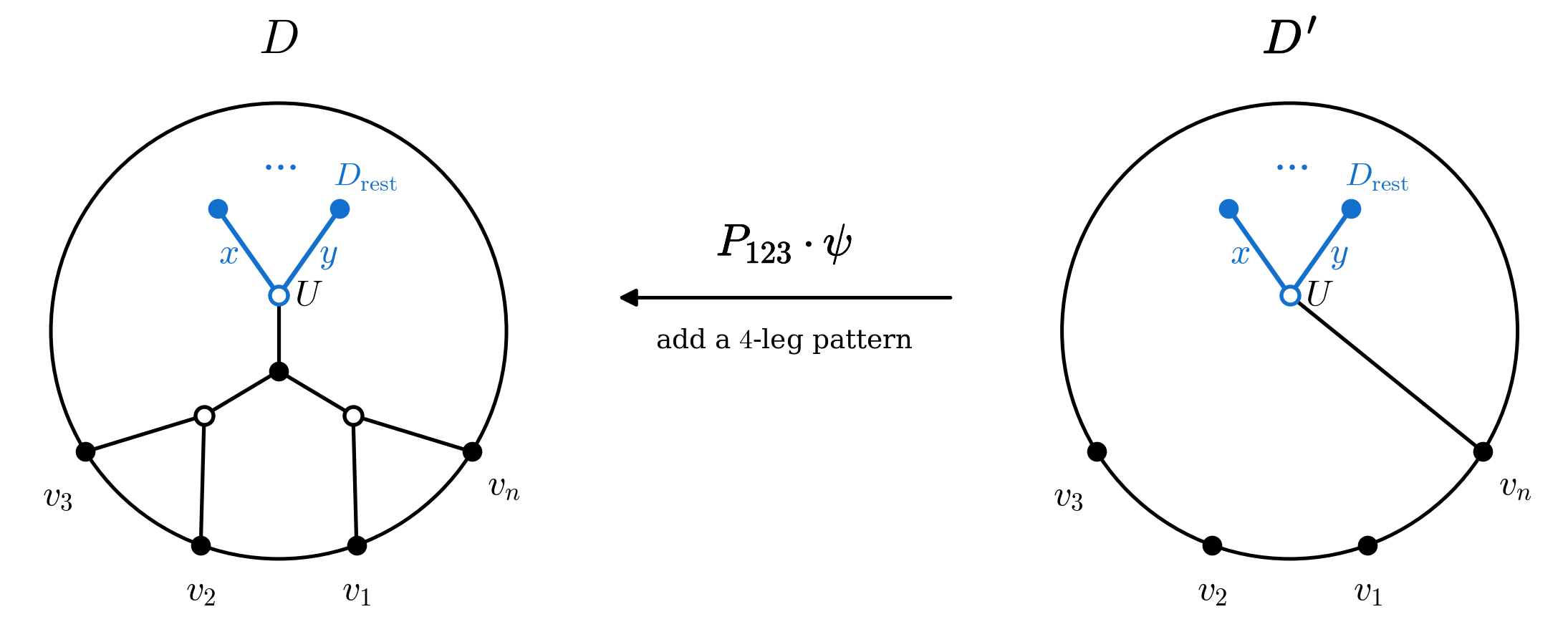}
    \caption{Illustration of the correspondence between $\tpsi$ and adding a 4-leg pattern.}
    \label{applying tpsi = applying alpha figure}
\end{figure}
\end{theorem}
\begin{proof}

Write $D = \Lambda \sqcup D_{\text{rest}}$ and $ D' = L_{v_nU}\sqcup D_{\text{rest}}$. Here $\Lambda$ is the 4-leg pattern in $D$, $L_{v_nU}$ is the line from $v_n$ to $U$ in $D'$, and $\sqcup$ means that two subgraphs are disjoint in edges and the only intersection is at the one vertex connecting them. 

Orient all edges toward vertex $U$ and evaluate two webs at $U$. Let the two incoming legs from $D_{\text{rest}}$ be $\vec{x}$ and $\vec{y}$. 

In $D'$, the third incoming vector is just $\vec{v_n}$, so $[D'] = \det(\vec{v_n}, \vec{x}, \vec{y}).$ 

In $D$, two incoming edges $D_{\text{rest}}$ are still $\vec{x}$ and $\vec{y}$, and the third incoming edge is $\vec{b} = b(\vec{v_n} \times \vec{v_1}, \vec{v_{2}} \times \vec{v_3})$. Evaluating at $U$, $[D] = \det(\vec{b}, \vec{x}, \vec{y}).$

\begin{figure}[htbp]
    \centering
    \includegraphics[width=0.8\textwidth]{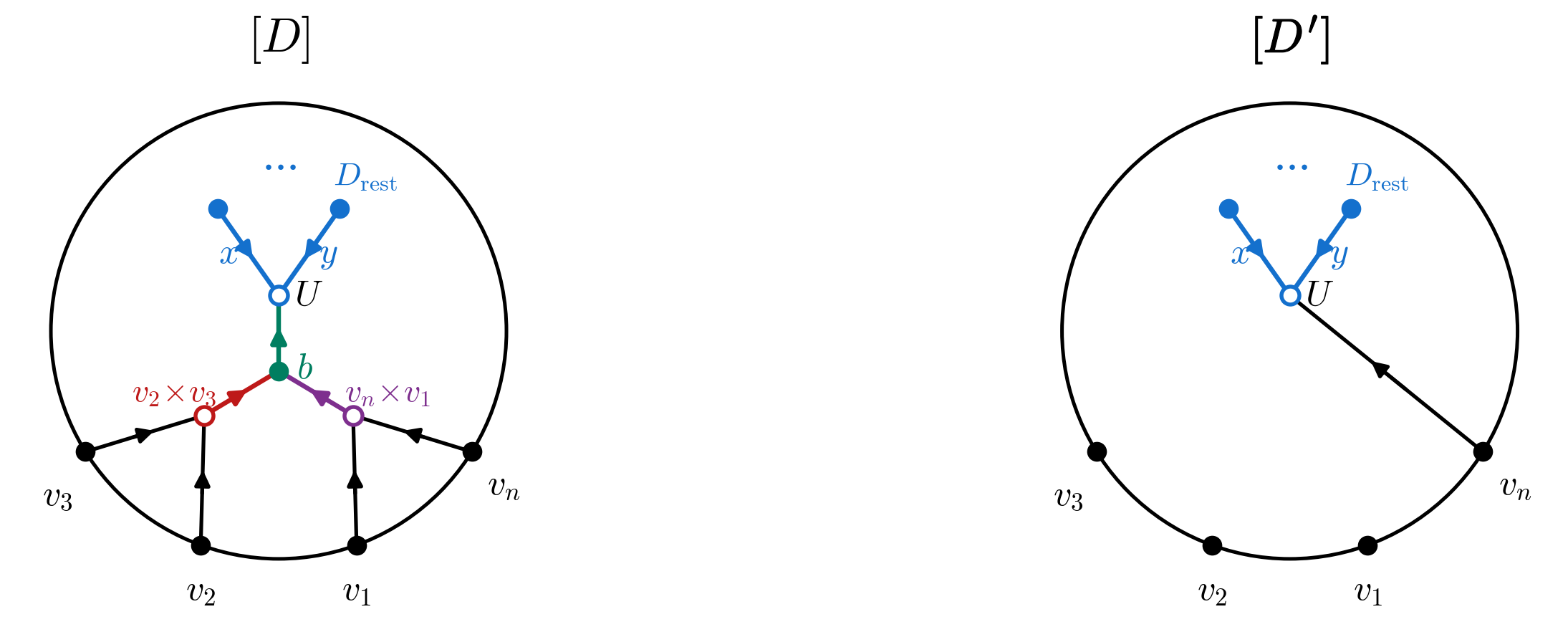}
    \caption{Evaluate $[D]$ and $[D']$ at vertex $U$.}
    \label{evaluate web figure}
\end{figure}

We want to show $D = P_{123} \cdot \psi(D')$, that is, $$\det(\vec{b}, \vec{x}, \vec{y}) = P_{123} \cdot \psi(\det(\vec{v_n}, \vec{x}, \vec{y})) = \det (P_{123} \cdot \psi (\vec{v_n}), \psi(\vec{x}), \psi(\vec{y})).$$

Because $\vec{x}$ and $\vec{y}$ does not contain $\vec{v_n}, \vec{v_1}, \vec{v_2}, \vec{v_3}$ as a factor, $\psi (\vec{x}) = \vec{x}$ and $\psi (\vec{y}) = \vec{y}$. We are left to show that $P_{123} \cdot \psi(\vec{v_n}) = \vec{b}$.

We know from the cluster quasihomomorphism that 
$$P_{123} \cdot \psi(\vec{v_n}) = P_{23n} \vec{v_1} - P_{123} \vec{v_n} = \det(\vec{v_2}, \vec{v_3}, \vec{v_{n}}) \vec{v_{1}} - \det(\vec{v_1}, \vec{v_2}, \vec{v_3})\vec{v_n}.$$

By \cref{multi linear identity}, $$b(\vec{v_{n}} \times \vec{v_1}, \vec{v_2} \times \vec{v_3}) = \det(\vec{v_2}, \vec{v_3}, \vec{v_{n}}) \vec{v_1} - \det(\vec{v_1}, \vec{v_2}, \vec{v_3}) \vec{v_n}.$$
Hence $P_{123} \cdot \psi(\vec{v_n}) = \vec{b}$ as desired. 
\end{proof}

\begin{theorem} \label{thm:main1}
If an $\SL_3$ tensor diagram $D$ for $\Gr(3,n)$ is a tree web in the sense of \cref{tree web def}, then $[D]$ is a cluster or coefficient variable for $\Gr(3,n)$.
\end{theorem} 

\begin{proof}
 Without loss of generality, by \cref{lem:add-marker-embed-gr}, we can assume that all boundary vertices are connected to the rest of the web.
 
We use induction on the number $n$ of boundary vertices.  By \cref{3m boundary vertices}, $n$ is a multiple of 3. 

A tree web with 3 boundary vertices corresponds to a Pl\"ucker coordinate $P_{abc}$, which is a cluster or coefficient variable. Assume the proposition holds for tree webs with $n-3$ boundary vertices.

By \cref{inductive generation prop}, $D$ is obtained by adding a $4$-leg pattern to a planar tree web $D'$ on $n-3$ boundary vertices.  Without loss of generality -- since cyclic shifts and reflections preserve the cluster algebra structure -- we assume the 4-leg pattern added for $D$ is on boundary vertices $n,1,2,3$.

By \cref{adding 4-leg is applying psi}, adding a 4-leg pattern to $D'$ corresponds to the map $\psi \cdot P_{123}$ to $[D']$. By the inductive hypothesis, $[D']$ is a cluster or coefficient variable; since $\psi$ is a cluster quasihomomorphism and $P_{123}$ is a frozen variable, $[D]$ is a cluster or coefficient variable (potentially multiplied by a frozen factor, that is a Laurent monomial in frozen variables).  

However, by \cite[Lemma 12.2]{FP16}, a planar tree web must be irreducible, so there is no extra frozen factor. 

Alternatively, we can prove directly that $[D]$ is irreducible. If not, we can write $[D]=[E][F]$, where $F$ is the web of a single frozen variable, and $E$ is another tensor diagram. Since $D$ is a tree web, the boundary vertices of $D$ have degree at most 1, and so is the superposition of $E$ and $F$ (because the tuple of degrees of boundary vertices, called the \emph{multidegree}, is determined by $[D]$ and independent of chosen diagrammtic representations, see \cite[Equation 4.3]{FP16}). Because $F$ is a frozen variable, it is a tripod with consecutive boundary vertices, say $a,a+1, a+2$. Then the boundary vertices of $E$ are supported on the complementary vertices $a+3,\dots,n,1,2,\dots,a-1$. But then $F$ can be drawn in an arbitrarily small neighborhood close to the boundary vertices $a,a+1, a+2$ so that $E$ and $F$ are disjoint, and there are no skein relations we can apply to turn their superposition into one connected component (namely the tree web $D$). This is a contradiction.
\end{proof}

\section{Every $\Gr(4,n)$ tree web represents a cluster variable}

The goal of this section is to prove \cref{thm:main2}, which says that  if a $\Gr(4,n)$ tensor diagram $D$ is a tree web, then 
$[D]$ is a cluster or coefficient variable for $\C[\widehat{\Gr}(4,n)]$.

To prove \cref{thm:main2}, we will show that every $\Gr(4,n)$ tree web can be created from a \emph{tetrapod} (which corresponds to a Pl\"ucker coordinate $P_{abcd}$) by using four local operations. We will then show that each of these operations correspond to applying a cluster quasi-homomorphism to the corresponding web invariant.

\subsection{Every $\Gr(4,n)$ tree web can be constructed from four local operations}

Recall the definition of $\Gr(4,n)$ webs from Definition \ref{Gr(4,n) web def}.

\begin{proposition}\label{prop:decompose}
Every $\Gr(4,n)$ tree web can be inductively generated from a tetrapod by using the four types of operations in \cref{fig:promotions_for_Gr4n}. 
\end{proposition}

\begin{figure}[htbp]
    \centering
    \includegraphics[width=\textwidth]{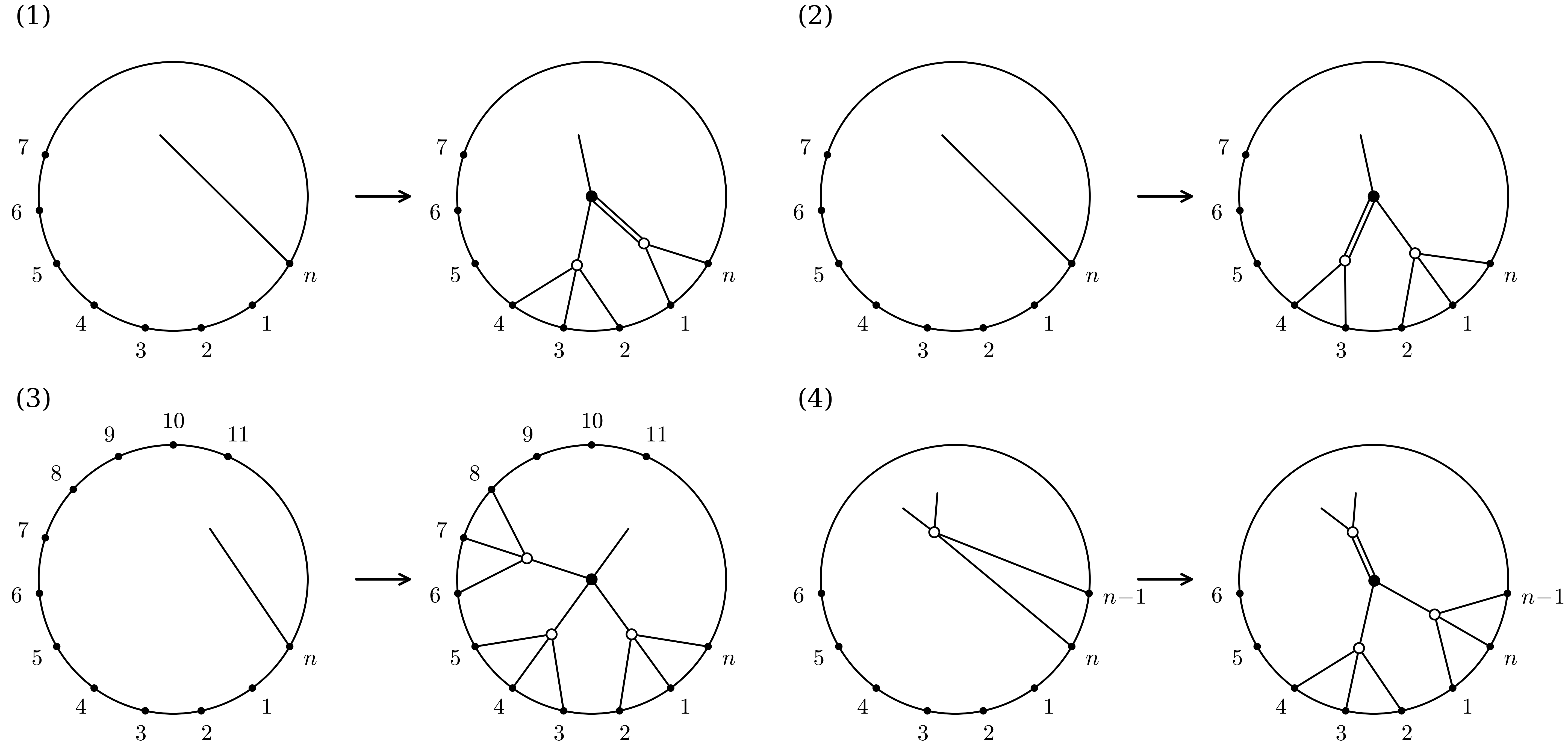}
\caption{(1) and (2): adding a 5-leg pattern; 
(3): adding a 9-leg pattern; 
(4): adding a 6-leg pattern.
}
\label{fig:promotions_for_Gr4n}
\end{figure}

\begin{proof}
Let $D$ be a $\Gr(4,n)$ tree web. Then the number $n$ of boundary vertices must be a multiple of $4$.
Choose any internal white vertex: we regard it as a root vertex and denote it $w_{\rm root}$.  Directing all edges towards it, we can regard the tree web diagram as a rooted tree.  We use induction on $n$ to prove the result.
  
If $n=4$, then $D$ must be a tetrapod, and we are done.  Otherwise, if $n>4$, then we can choose a black internal vertex $v^*$ which has maximum distance from $w_{\rm root}$. 

Let $e$ be the edge incident to $v^*$ which is directed towards $w_{\rm root}$.  If $e$ has mutiplicity $2$, then $v^*$ has two white children $w_1$ and $w_2$.  Both $w_1$ and $w_2$ must have all of their children being boundary vertices; otherwise $v^*$ would not have maximum distance from the root.
Thus $w_1$ and $w_2$ both have three black boundary children, and the local configuration is a 6-leg pattern, as in \Cref{fig:promotions_for_Gr4n}, Case (4).  Thus, $D$ can be obtained from a $\C[\widehat{\Gr}(4,n-4)]$ planar tree web by adding a 6-leg pattern.

On the other hand, if $e$ has multiplicity $1$, then $v^*$ has 
either three white children or two white children.  If it has three white children then as before, each of them must have all of their children being boundary vertices, and the local configuration is a 9-leg pattern, as in 
\Cref{fig:promotions_for_Gr4n}, Case (3).  Thus, $D$ can be obtained from a  $\C[\widehat{\Gr}(4,n-8)]$  planar tree web by adding a 9-leg pattern.

If $v^*$ has two white children, then one of them is connected to $v^*$ via an edge of multiplicity $1$, and the other is connected via an edge of multiplicity $2$.  In these cases, the local configuration must be a 5-leg pattern, as in 
\Cref{fig:promotions_for_Gr4n}, Cases (1) or (2).
Thus, $D$ can be obtained from a  $\Gr(4,n-4)$ planar tree web by adding a 9-leg pattern.
\end{proof}

\subsection{The local operations correspond to cluster quasi-homomorphisms}

\begin{lemma}\label{lem:case1}
Adding a 5-leg pattern   as in Case (1) of \Cref{fig:promotions_for_Gr4n} corresponds to the map 
$\CC[\widehat{\Gr}(4,[5,n])] \to \C[\widehat{\Gr}(4,n)]$ (called \emph{upper promotion} in \cite{even2023cluster}
and \emph{unary star promotion} in \cite{plabictangle}) which is defined as follows.
We first consider the  map $\Gr(4,n) \to \Gr(4, [5,n])$ which sends the $4 \times n$ matrix $Z=(Z_1\mid \ldots \mid Z_n)$ to the 
$4 \times (n-4)$ matrix $Z' = (Z'_5\mid \ldots \mid Z'_n)$, where $Z'_i=Z_i$ for $i \ne n$ and $Z'_n = Z_n + \frac{P_{2,3,4,n}}{P_{1,2,3,4}} Z_{1}$.
This map induces a $\C$-algebra homomorphism $\psi: \CC[\widehat{\Gr}(4,[5,n])] \to \C[\widehat{\Gr}(4,n)]$ which maps Pl\"ucker coordinates as follows:
\begin{align}
P_I \mapsto
\begin{cases}
P_I,
& \text{if } I \cap \{1,n\} \neq \{n\}, \\
P_I - \frac{P_{2,3,4,n}}{P_{1,2,3,4}} P_{(\{1\} \cup I\setminus \{n\})},
& \text{if } I \cap \{1,n\} = \{n\}.
\end{cases}
\end{align} 
This map is a cluster quasihomomorphism.
\end{lemma}
\begin{proof}
At the left-hand side of Case (1) of \Cref{fig:promotions_for_Gr4n}, the edge adjacent to 
the boundary vertex $n$ represents the vector $Z_n$.
Let $\ast$ denote the shuffle product
as in \cite[Definition 2.13]{plabictangle}.
The vector corresponding to the ``hanging edge'' coming off of the internal black vertex 
at the right-hand side of Case (1)
is $(n1) \ast (234)$, which is shorthand for 
\begin{align*}
(Z_n \wedge Z_1) \ast (Z_2 \wedge Z_3 \wedge Z_4) &= 
Z_n (Z_1 \wedge Z_2 \wedge Z_3 \wedge Z_4) - Z_1 (Z_n \wedge Z_2 \wedge Z_3 \wedge Z_4)\\
&= 
Z_n P_{1234} - Z_1 P_{n234} = Z_n P_{1234}+Z_1 P_{234n}
\end{align*}
In order to induce a map 
$\CC[\widehat{\Gr}(4,[5,n])] \to \C[\widehat{\Gr}(4,n)])$ which preserves degree, we rescale this vector by $P_{1234}$: this results in the map described in the lemma.
Finally, the fact that this map is a cluster algebra quasihomomorphism is proved in 
\cite[Theorem 8.1]{plabictangle}.
\end{proof}

\begin{lemma}\label{lem:case2}
Adding a 5-leg pattern as in Case (2) of \Cref{fig:promotions_for_Gr4n} corresponds to a map 
$\CC[\widehat{\Gr}(4,[5,n])] \to \C[\widehat{\Gr}(4,n)]$, which is a cluster quasihomomorphism.
\end{lemma}
\begin{proof}
As one can see from \Cref{fig:promotions_for_Gr4n}, Case (2) is obtained from Case (1) by 
a reflection as well as a rotation of the indices. 
By \cref{cor:cycrefl}, 
 reflecting and rotating indices preserve the cluster algebra structure, so by \cref{lem:case1}, we are done.
\end{proof}

In what follows, we use the notation  $\langle ab|cde|fgh \rangle:=
    P_{abcd}P_{efhg} - P_{abce}P_{dfgh} + P_{abde}P_{cfgh}.$
    If $f$ is a regular function, we also use the notation
    $ \C[\widehat{\Gr}(4,n)][f^{-1}]$ to denote the localization of 
     $\C[\widehat{\Gr}(4,n)]$ at the function $f$. In \cref{lem:9leg}, $f$ will be a cluster variable, which is mutable in the usual cluster structure on the Grassmannian; localizing at $f$ corresponds to freezing this cluster variable.

\begin{lemma}\label{lem:9leg}
Adding a 9-leg pattern as in Case (3) of \Cref{fig:promotions_for_Gr4n} corresponds to a map 
$\CC[\widehat{\Gr}(4,[9,n])] \to \C[\widehat{\Gr}(4,n)]$ (called \emph{unary spurion promotion} in \cite[Definition 5.5]{plabictangle}) which is defined as follows.
We first consider the  map $\Gr(4,n) \to \Gr(4, [9,n])$ which sends the $4 \times n$ matrix $Z=(Z_1\mid \ldots \mid Z_n)$ to the 
$4 \times (n-8)$ matrix $Z' = (Z'_5\mid \ldots \mid Z'_n)$, where $Z'_i=Z_i$ for $i \ne n$ and 
\begin{equation}\label{eq:Z2}
Z'_n = \frac{1}{{\langle 12 |345 | 678 \rangle}}
(Z_n \wedge Z_1 \wedge Z_2) \ast 
(Z_3 \wedge Z_4 \wedge Z_5) \ast(Z_6 \wedge Z_7 \wedge Z_8).
\end{equation}
This map induces a $\C$-algebra homomorphism\footnote{This map can be written out explicitly in terms of Pl\"ucker coordinates but it is not needed here.} $\psi: \CC[\widehat{\Gr}(4,[9,n])] \to \C[\widehat{\Gr}(4,n)][{\langle 12 |345 | 678 \rangle}]^{-1}$ which is a cluster quasihomomorphism.
\end{lemma}
\begin{proof}
At the left-hand side of Case (3) of \Cref{fig:promotions_for_Gr4n}, the edge adjacent to 
the boundary vertex $n$ represents the vector $Z_n$.
The vector corresponding to the ``hanging edge'' coming off of the internal black vertex 
at the right-hand side of Case (3)
is 
\begin{equation*}
(Z_n \wedge Z_1 \wedge Z_2) \ast 
(Z_3 \wedge Z_4 \wedge Z_5) \ast(Z_6 \wedge Z_7 \wedge Z_8),
\end{equation*}
which can be written as a linear combination of the vectors $Z_n$, $Z_1$, and $Z_2$.
If we then rescale by  a constant to make the coefficient of $Z_n$ be $1$, we get 
    $Z_n -Z_1 \frac{\langle n2 | 345 | 678\rangle}{\langle 12 | 345 | 678 \rangle} + Z_2 \frac{\langle n1 | 345 | 678 \rangle}{\langle 12 |345 | 678 \rangle},$ which is equal to the quantity in \eqref{eq:Z2}.

    The above map induced by sending $Z$ to $Z'$ is a special case of the spurion promotion which was defined in \cite[Definition 5.5]{plabictangle}. The fact that it is a cluster algebra quasihomomorphism was proved in  \cite[Theorem 8.7]{plabictangle}.
\end{proof}

\begin{lemma}\label{lem:6leg}
Adding a 6-leg pattern   as in Case (4) of \Cref{fig:promotions_for_Gr4n} corresponds to a map 
$\CC[\widehat{\Gr}(4,[5,n])]\to \C[\widehat{\Gr}(4,n)]$ which is 
defined as follows.
We first consider the  map $\Gr(4,n) \to \Gr(4, [5,n])$ which sends the $4 \times n$ matrix $Z=(Z_1\mid \ldots \mid Z_n)$ to the 
$4 \times (n-4)$ matrix $Z' = (Z'_5\mid \ldots \mid Z'_n)$, where 
\begin{equation}\label{eq:6first}
Z'_j = -\Bigl(Z_j+\frac{\Pl{2,3,4,j}}{\Pl{1234}}\,Z_1\Bigr)
\quad \text{ for }j=n-1,n,
\end{equation}
and $Z'_i=Z_i$ otherwise.
Thus
\begin{equation}\label{eq:6}
Z'_{n-1}\wedge Z'_n = (Z_{n-1} \wedge Z_n)  - (Z_{1} \wedge Z_{n-1}) \frac{P_{2,3,4,n}}{P_{1234}} + (Z_1 \wedge Z_n) \frac{P_{2,3,4,n-1}}{P_{1234}}.
\end{equation}
Let 
\begin{align*}
\Rv{j_0}{ijk}&=\Pl{2,3,4,j_0}\Pl{1ijk}-\Pl{1234}\Pl{ijk\,j_0},\\
\Qv{ij}&=\Pl{1234}\Pl{ij,n-1,n}-\Pl{2,3,4,n}\Pl{1ij,n-1}+\Pl{2,3,4,n-1}\Pl{1ij,n}.
\end{align*}
This map induces a $\CC$-algebra homomorphism 
$\Phi_n: \CC[\widehat{\Gr}(4,[5,n])]\to \C[\widehat{\Gr}(4,n)]$ which maps Pl\"ucker coordinate as follows:
for $S\subseteq\{5,\dots,n\}$
\begin{equation}\label{eq:6Plucker}
\Phi_n(\Pl{S}')=
\begin{cases}
\Pl{S}, & S\cap\{n-1,n\}=\emptyset,\\[3pt]
\Rv{j_0}{ijk}\big/\Pl{1234}, & S=\{i,j,k,j_0\},\ j_0\in\{n-1,n\},\\[3pt]
\Qv{ij}\big/\Pl{1234}, & S=\{i,j,n-1,n\}.
\end{cases}
\end{equation}

\end{lemma}

\begin{proof}

Note that the $6$-leg pattern 
in Case (4) of \Cref{fig:promotions_for_Gr4n} represents the
$2$-vector
\[
(Z_{n-1} \wedge Z_n \wedge Z_1)\ast(Z_2 \wedge Z_3 \wedge Z_4).
\]
Expanding gives
\begin{align*}
    (Z_{n-1} \wedge Z_n) P_{1234} &- (Z_{n-1} \wedge Z_1) P_{n,2,3,4} + (Z_n \wedge Z_1) P_{n-1,2,3,4}\\
    &= 
      (Z_{n-1} \wedge Z_n) P_{1234} - (Z_{1} \wedge Z_{n-1}) P_{2,3,4,n} + (Z_1 \wedge Z_n) P_{2,3,4,n-1}
\end{align*}
Now if we rescale by 
$P_{1234}$, we get \eqref{eq:6}.
It is easy to see that \eqref{eq:6Plucker} follows from \eqref{eq:6}
\end{proof}

We will prove that this map
is a cluster algebra quasihomomorphism in \cref{sec:6leg}.

\subsection{The 6-leg pattern map is a cluster quasihomomorphism}\label{sec:6leg}

\begin{theorem}\label{thm:6leg}
    The 6-leg pattern map from 
    \cref{lem:6leg} is a cluster quasihomomorphism.
\end{theorem}

Before proving the theorem, we examine the cases
$n=10$ and $n=11$ to gain intuition.

\subsection*{The case $n=10$.}
We start with the rectangles seed $\Sigma$ for $\Gr(4,\{5,\dots,10\})\cong\Gr(4,6)$, see \cref{fig:rect1}. We indicate the frozen nodes in gray and the mutable nodes in red.

\begin{figure}[h]
\begin{center}
\resizebox{0.46\linewidth}{!}{%
\begin{tikzpicture}[
  >=Stealth,
  mut/.style={draw=srcred, text=srcred, rounded corners, align=center, inner sep=3pt, minimum width=1.9cm, minimum height=0.72cm, font=\scriptsize\bfseries, fill=white},
  frz/.style={draw, rounded corners, align=center, inner sep=3pt, minimum width=1.9cm, minimum height=0.72cm, font=\scriptsize, fill=gray!20},
  arr/.style={->, line width=0.55pt}
]
\node[frz] (E) at (-2.9,0.9) {$\Pl{5678}$};
\node[mut] (R11) at (0,0)    {$\Pl{5679}$};
\node[frz] (R12) at (3.0,0)  {$\Pl{567\,10}$};
\node[mut] (R21) at (0,-1.7) {$\Pl{5689}$};
\node[frz] (R22) at (3.0,-1.7){$\Pl{569\,10}$};
\node[mut] (R31) at (0,-3.4) {$\Pl{5789}$};
\node[frz] (R32) at (3.0,-3.4){$\Pl{589\,10}$};
\node[frz] (R41) at (0,-5.1) {$\Pl{6789}$};
\node[frz] (R42) at (3.0,-5.1){$\Pl{789\,10}$};
\begin{pgfonlayer}{background}
\draw[arr] (E) -- (R11);   \draw[arr] (R11) -- (R12); \draw[arr] (R11) -- (R21);
\draw[arr] (R22) -- (R11); \draw[arr] (R21) -- (R22); \draw[arr] (R21) -- (R31);
\draw[arr] (R32) -- (R21); \draw[arr] (R31) -- (R32); \draw[arr] (R31) -- (R41);
\draw[arr] (R42) -- (R31);
\end{pgfonlayer}
\end{tikzpicture}
}
\end{center}
\caption{The rectangles seed $\Sigma$ for 
$\Gr(4,\{5,6,\dots,10\}$\label{fig:rect1}}
\end{figure}
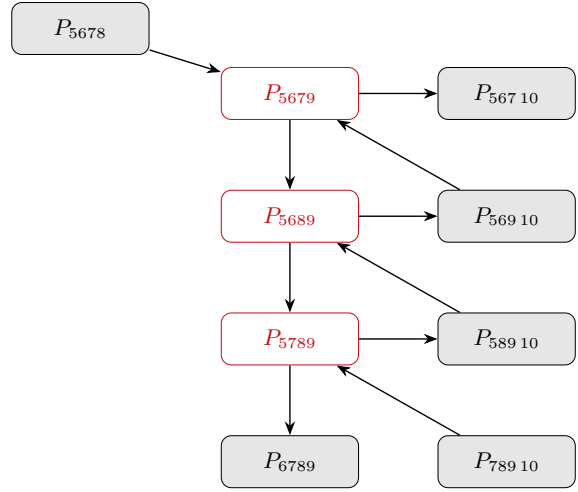

When we apply the map $\Phi_n$ to $\Sigma$, we get the following (where the mutable cluster variables are in red): 
\[
\begin{aligned}
{\color{srcred}\Pl{5679}}&\mapsto {\color{srcred}\Rv{9}{567}}/\Pl{1234}, &
{\color{srcred}\Pl{5689}}&\mapsto {\color{srcred}\Rv{9}{568}}/\Pl{1234}, &
{\color{srcred}\Pl{5789}}&\mapsto {\color{srcred}\Rv{9}{578}}/\Pl{1234},\\
{\color{black}\Pl{5678}}&\mapsto {\color{black}\Pl{5678}}, &
\Pl{6789}&\mapsto \Rv{9}{678}/\Pl{1234},&
\Pl{567\,10}&\mapsto \Rv{10}{567}/\Pl{1234}, \\
\Pl{569\,10}&\mapsto \Qv{56}/\Pl{1234}, &
\Pl{589\,10}&\mapsto \Qv{58}/\Pl{1234}, &
\Pl{789\,10}&\mapsto \Qv{78}/\Pl{1234}.
\end{aligned}
\]

\cref{fig:imagerect1} shows a seed which 
contains the image of the rectangles cluster for 
$\Gr(4,\{5,\dots,10\})$. Frozen variables are colored in gray; the images under $\Phi_n$ of the initial cluster variables are marked in bold; and the images of the initial mutable cluster variables are colored in red.  Clearly properties (1) and (3) of \cref{def:quasi} are satisfied for the rectangles seed $\Sigma$ and the new seed 
$\overline{\Sigma}$.  In particular, the induced quiver on the images of the mutable cluster variables is isomorphic to the induced quiver on the original mutable cluster variables, and this isomorphism takes
each mutable cluster variable $x$ to $\Phi_n(x)$ (up to a frozen factor). 
Note that the unique extra arrow at an image vertex is
$\Pl{1234}\to{\color{srcred}\Rv{9}{567}}$.
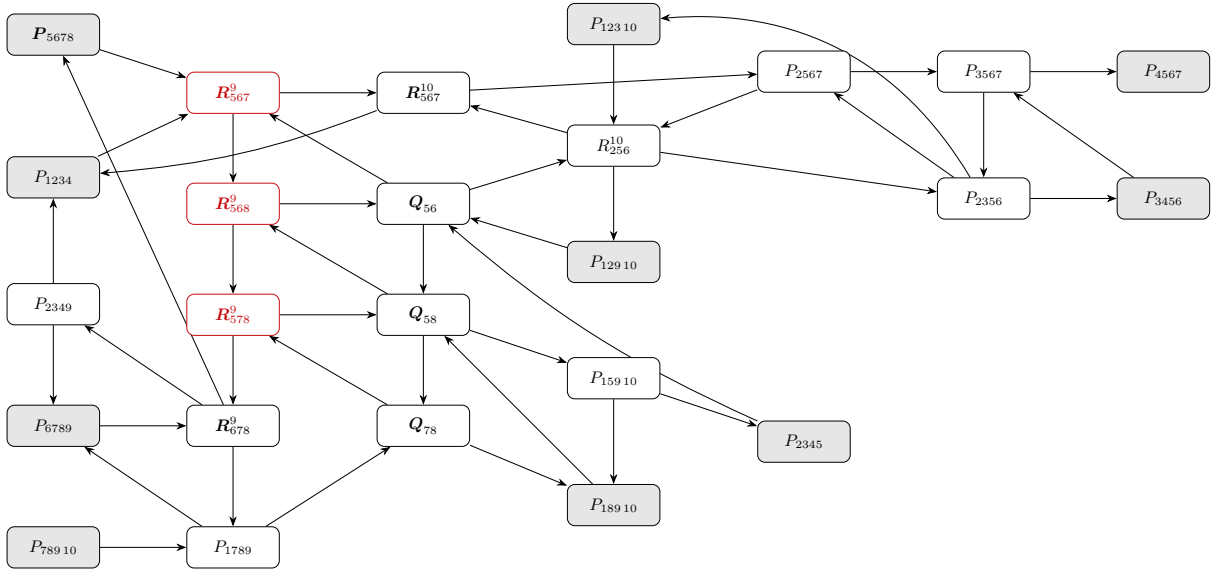
\begin{figure}
\begin{center}
\resizebox{0.97\linewidth}{!}{%
\begin{tikzpicture}[
  >=Stealth,
  mut/.style={draw, rounded corners, align=center, inner sep=3pt, minimum width=1.75cm, minimum height=0.78cm, font=\scriptsize, fill=white},
  imgmut/.style={draw, rounded corners, align=center, inner sep=3pt, minimum width=1.75cm, minimum height=0.78cm, font=\scriptsize\bfseries, fill=white},
  redmut/.style={draw=srcred, text=srcred, rounded corners, align=center, inner sep=3pt, minimum width=1.75cm, minimum height=0.78cm, font=\scriptsize\bfseries, fill=white},
  frz/.style={draw, rounded corners, align=center, inner sep=3pt, minimum width=1.75cm, minimum height=0.78cm, font=\scriptsize, fill=gray!20},
  imgfrz/.style={draw, rounded corners, align=center, inner sep=3pt, minimum width=1.75cm, minimum height=0.78cm, font=\scriptsize\bfseries, fill=gray!20},
  arr/.style={->, line width=0.55pt}
]
\node[imgfrz] (P5678)  at (-3.4, 1.1)  {$\BPl{5678}$};
\node[redmut] (R9567)  at (0,0)        {$\BRv{9}{567}$};
\node[imgmut] (R10567) at (3.6,0)      {$\BRv{10}{567}$};
\node[redmut] (R9568)  at (0,-2.1)     {$\BRv{9}{568}$};
\node[imgmut] (Q56)    at (3.6,-2.1)   {$\BQv{56}$};
\node[redmut] (R9578)  at (0,-4.2)     {$\BRv{9}{578}$};
\node[imgmut] (Q58)    at (3.6,-4.2)   {$\BQv{58}$};
\node[imgmut] (R9678)  at (0,-6.3)     {$\BRv{9}{678}$};
\node[imgmut] (Q78)    at (3.6,-6.3)   {$\BQv{78}$};
\node[frz] (P1234)  at (-3.4,-1.6)  {$\Pl{1234}$};
\node[mut] (P2349)  at (-3.4,-4.0)  {$\Pl{2349}$};
\node[frz] (P6789)  at (-3.4,-6.3)  {$\Pl{6789}$};
\node[frz] (P78910) at (-3.4,-8.6)  {$\Pl{789\,10}$};
\node[mut] (P1789)  at (0,-8.6)     {$\Pl{1789}$};
\node[mut] (P15910) at (7.2,-5.4)   {$\Pl{159\,10}$};
\node[frz] (P18910) at (7.2,-7.8)   {$\Pl{189\,10}$};
\node[frz] (P2345)  at (10.8,-6.6)  {$\Pl{2345}$};
\node[mut] (R10256) at (7.2,-1.0)   {$\Rv{10}{256}$};
\node[frz] (P12310) at (7.2, 1.3)   {$\Pl{123\,10}$};
\node[frz] (P12910) at (7.2,-3.2)   {$\Pl{129\,10}$};
\node[mut] (P2567)  at (10.8, 0.4)  {$\Pl{2567}$};
\node[mut] (P3567)  at (14.2, 0.4)  {$\Pl{3567}$};
\node[frz] (P4567)  at (17.6, 0.4)  {$\Pl{4567}$};
\node[mut] (P2356)  at (14.2,-2.0)  {$\Pl{2356}$};
\node[frz] (P3456)  at (17.6,-2.0)  {$\Pl{3456}$};
\begin{pgfonlayer}{background}
\draw[arr] (P5678)  -- (R9567);
\draw[arr] (R9567)  -- (R10567);
\draw[arr] (R9567)  -- (R9568);
\draw[arr] (Q56)    -- (R9567);
\draw[arr] (R9568)  -- (Q56);
\draw[arr] (R9568)  -- (R9578);
\draw[arr] (Q58)    -- (R9568);
\draw[arr] (R9578)  -- (Q58);
\draw[arr] (R9578)  -- (R9678);
\draw[arr] (Q78)    -- (R9578);
\draw[arr] (Q56)    -- (Q58);
\draw[arr] (Q58)    -- (Q78);
\draw[arr] (P1234)  -- (R9567);
\draw[arr] (R10567) to[bend left=8] (P1234);
\draw[arr] (P2349)  -- (P1234);
\draw[arr] (P2349)  -- (P6789);
\draw[arr] (R9678)  -- (P2349);
\draw[arr] (R9678)  -- (P5678);
\draw[arr] (R9678)  -- (P1789);
\draw[arr] (P6789)  -- (R9678);
\draw[arr] (P1789)  -- (P6789);
\draw[arr] (P78910) -- (P1789);
\draw[arr] (P1789)  -- (Q78);
\draw[arr] (Q78)    -- (P18910);
\draw[arr] (P18910) -- (Q58);
\draw[arr] (Q58)    -- (P15910);
\draw[arr] (P15910) -- (P18910);
\draw[arr] (P15910) -- (P2345);
\draw[arr] (P2345)  to[bend left=8] (Q56);
\draw[arr] (Q56)    -- (R10256);
\draw[arr] (R10256) -- (R10567);
\draw[arr] (R10256) -- (P12910);
\draw[arr] (P12910) -- (Q56);
\draw[arr] (P12310) -- (R10256);
\draw[arr] (R10567) -- (P2567);
\draw[arr] (P2567)  -- (R10256);
\draw[arr] (R10256) -- (P2356);
\draw[arr] (P2567)  -- (P3567);
\draw[arr] (P3567)  -- (P4567);
\draw[arr] (P3567)  -- (P2356);
\draw[arr] (P2356)  -- (P2567);
\draw[arr] (P2356)  -- (P3456);
\draw[arr] (P3456)  -- (P3567);
\draw[arr] (P2356)  to[bend right=32] (P12310);
\end{pgfonlayer}
\end{tikzpicture}
}
\end{center}
\caption{A seed $\overline{\Sigma}$ which contains the image of the seed from \cref{fig:rect1}. \label{fig:imagerect1}}
\end{figure}

We now check that 
 property (2) of \cref{def:quasi} holds.
We have that 
\begin{align*}
\hat{y}_{\Sigma}(P_{5679}) &=\frac{P_{5689}P_{567\, 10}}{P_{5678}P_{569\, 10}} \qquad \text{ so }\Phi_n(\hat{y}_{\Sigma}(P_{5679})) =\frac{R^9_{568} R^{10}_{567}}{P_{5678} Q_{56} P_{1234}}\\
\hat{y}_{\Sigma}(P_{5689}) &=\frac{P_{5789}P_{569\, 10}}{P_{5679}P_{589\, 10}}  
\qquad \text{ so }\Phi_n(\hat{y}_{\Sigma}(P_{5689})) =\frac{R^9_{578} Q_{56}}{R^9_{567} Q_{58}}\\
\hat{y}_{\Sigma}(P_{5789}) &=\frac{P_{6789}P_{589\, 10}}{P_{5689}P_{789\, 10}} \qquad \text{ so }\Phi_n(\hat{y}_{\Sigma}(P_{5789})) =\frac{R^9_{678} Q_{58}}{R^9_{568} Q_{78}}\\
    \end{align*}
These last three quantities agree with the exchange ratios 
$\hat{y}_{\overline{\Sigma}}(R^9_{567})$,
$\hat{y}_{\overline{\Sigma}}(R^9_{568})$, and
$\hat{y}_{\overline{\Sigma}}(R^9_{578}).$

To obtain the seed $\overline{\Sigma}$ in \cref{fig:imagerect1}, we start with the rectangles seed of $\Gr(4,10)$ in the cyclic order $5,6,7,8,9,10,1,2,3,4$. Its frozen
variables are the ten cyclic intervals and its mutable vertices (named for the cluster variables that are initially associated to them) are
\[
v_{5679},v_{567\,10},v_{1567},v_{2567},v_{3567},\;
v_{5689},v_{569\,10},v_{156\,10},v_{1256},v_{2356},\;
v_{5789},v_{589\,10},v_{159\,10},v_{125\,10},v_{1235}.
\]
We then apply the sequence of 13 mutations
\[
\mu_{v_{1235}},\mu_{v_{1256}},\mu_{v_{125\,10}},\mu_{v_{567\,10}},\mu_{v_{156\,10}},
\mu_{v_{1567}},\mu_{v_{5679}},\mu_{v_{569\,10}},\mu_{v_{567\,10}},\mu_{v_{5689}},
\mu_{v_{589\,10}},\mu_{v_{569\,10}},\mu_{v_{5789}}.
\]

\subsection*{The case $n=11$}

We start with the rectangles seed for  $\Gr(4,\{5,\dots,11\})\cong\Gr(4,7)$.

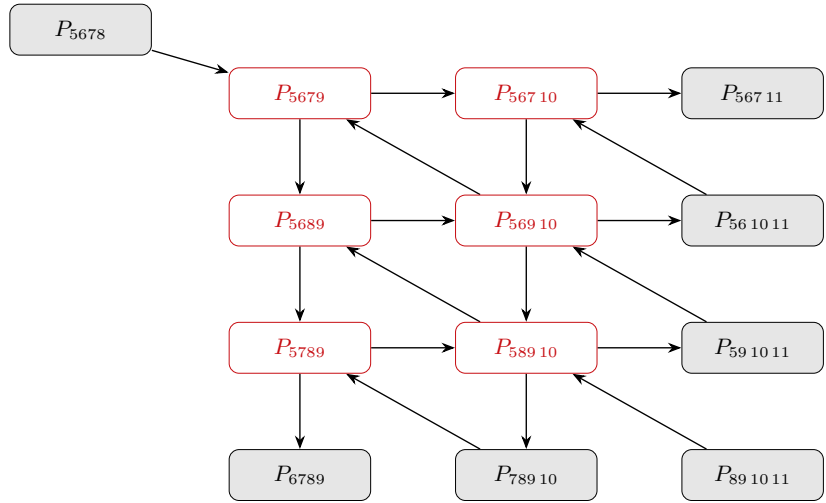
\begin{figure}[h]
\begin{center}
\resizebox{0.66\linewidth}{!}{%
\begin{tikzpicture}[
  >=Stealth,
  mut/.style={draw=srcred, text=srcred, rounded corners, align=center, inner sep=3pt, minimum width=2.0cm, minimum height=0.72cm, font=\scriptsize\bfseries, fill=white},
  frz/.style={draw, rounded corners, align=center, inner sep=3pt, minimum width=2.0cm, minimum height=0.72cm, font=\scriptsize, fill=gray!20},
  arr/.style={->, line width=0.55pt}
]
\node[frz] (E) at (-3.1,0.9) {$\Pl{5678}$};
\node[mut] (R11) at (0,0)     {$\Pl{5679}$};
\node[mut] (R12) at (3.2,0)   {$\Pl{567\,10}$};
\node[frz] (R13) at (6.4,0)   {$\Pl{567\,11}$};
\node[mut] (R21) at (0,-1.8)  {$\Pl{5689}$};
\node[mut] (R22) at (3.2,-1.8){$\Pl{569\,10}$};
\node[frz] (R23) at (6.4,-1.8){$\Pl{56\,10\,11}$};
\node[mut] (R31) at (0,-3.6)  {$\Pl{5789}$};
\node[mut] (R32) at (3.2,-3.6){$\Pl{589\,10}$};
\node[frz] (R33) at (6.4,-3.6){$\Pl{59\,10\,11}$};
\node[frz] (R41) at (0,-5.4)  {$\Pl{6789}$};
\node[frz] (R42) at (3.2,-5.4){$\Pl{789\,10}$};
\node[frz] (R43) at (6.4,-5.4){$\Pl{89\,10\,11}$};
\begin{pgfonlayer}{background}
\draw[arr] (E) -- (R11);
\draw[arr] (R11) -- (R12); \draw[arr] (R11) -- (R21);
\draw[arr] (R12) -- (R13); \draw[arr] (R12) -- (R22); \draw[arr] (R22) -- (R11);
\draw[arr] (R23) -- (R12);
\draw[arr] (R21) -- (R22); \draw[arr] (R21) -- (R31);
\draw[arr] (R22) -- (R23); \draw[arr] (R22) -- (R32); \draw[arr] (R32) -- (R21);
\draw[arr] (R33) -- (R22);
\draw[arr] (R31) -- (R32); \draw[arr] (R31) -- (R41);
\draw[arr] (R32) -- (R33); \draw[arr] (R32) -- (R42); \draw[arr] (R42) -- (R31);
\draw[arr] (R43) -- (R32);
\end{pgfonlayer}
\end{tikzpicture}
}
\end{center}
\caption{The rectangles seed $\Sigma$ for 
$\Gr(4,\{5,6,\dots,11\})$\label{fig:rect2}}
\end{figure}
When we apply $\Phi_n$ to this seed, we find the following:
\[
\begin{aligned}
\text{Five elements fixed by $\Phi_n$:}\quad &\Pl{5678},\ {\color{srcred}\Pl{5679}},\ {\color{srcred}\Pl{5689}},\
{\color{srcred}\Pl{5789}},\ \Pl{6789},\\
\text{Four $R^{10}$-type  variables:}\quad &{\color{srcred}\Pl{567\,10}}\mapsto{\color{srcred}\Rv{10}{567}}/\Pl{1234},\quad
{\color{srcred}\Pl{569\,10}}\mapsto{\color{srcred}\Rv{10}{569}}/\Pl{1234},\\
&{\color{srcred}\Pl{589\,10}}\mapsto{\color{srcred}\Rv{10}{589}}/\Pl{1234},\quad
\Pl{789\,10}\mapsto\Rv{10}{789}/\Pl{1234},\\
\text{One $R^{11}$-type  variable:}\quad &\Pl{567\,11}\mapsto\Rv{11}{567}/\Pl{1234},\\
\text{Three $Q$-type  variables:}\quad &\Pl{56\,10\,11}\mapsto\Qv{56}/\Pl{1234},\quad
\Pl{59\,10\,11}\mapsto\Qv{59}/\Pl{1234},\quad
\Pl{89\,10\,11}\mapsto\Qv{89}/\Pl{1234}.
\end{aligned}
\]
Note  that the three mutable variables in the leftmost column of $\Sigma$ are fixed by $\Phi_n$.

\cref{fig:imagerect2} shows a seed for $\Gr(4,11)$ that 
contains the image of the rectangles seed for 
$\Gr(4,\{5,\dots,11\})$. The images under $\Phi_n$ of the initial cluster variables are shown in bold, and the images of the initial mutable cluster variables colored in red.  
As before,  properties (1) and (3) of \cref{def:quasi} are clearly satisfied for the rectangles seed $\Sigma$ and the new seed 
$\overline{\Sigma}$.  In particular, the induced quiver on the images of the mutable cluster variables is isomorphic to the induced quiver on the original mutable cluster variables, and this isomorphism takes
each mutable cluster variable $x$ to $\Phi_n(x)$ (up to a frozen factor). 
We can also verify property (2) with a direct computation.

\begin{figure}
\begin{center}
\resizebox{0.99\linewidth}{!}{%
\begin{tikzpicture}[
  >=Stealth,
  mut/.style={draw, rounded corners, align=center, inner sep=3pt, minimum width=1.8cm, minimum height=0.78cm, font=\scriptsize, fill=white},
  imgmut/.style={draw, rounded corners, align=center, inner sep=3pt, minimum width=1.8cm, minimum height=0.78cm, font=\scriptsize\bfseries, fill=white},
  redmut/.style={draw=srcred, text=srcred, rounded corners, align=center, inner sep=3pt, minimum width=1.8cm, minimum height=0.78cm, font=\scriptsize\bfseries, fill=white},
  frz/.style={draw, rounded corners, align=center, inner sep=3pt, minimum width=1.8cm, minimum height=0.78cm, font=\scriptsize, fill=gray!20},
  imgfrz/.style={draw, rounded corners, align=center, inner sep=3pt, minimum width=1.8cm, minimum height=0.78cm, font=\scriptsize\bfseries, fill=gray!20},
  arr/.style={->, line width=0.55pt}
]
\node[imgfrz] (E)      at (-3.4, 1.1) {$\BPl{5678}$};
\node[redmut] (P5679)  at (0,0)       {$\RPl{5679}$};
\node[redmut] (R10567) at (3.6,0)     {$\RRv{10}{567}$};
\node[imgmut] (R11567) at (7.2,0)     {$\BRv{11}{567}$};
\node[redmut] (P5689)  at (0,-2.1)    {$\RPl{5689}$};
\node[redmut] (R10569) at (3.6,-2.1)  {$\RRv{10}{569}$};
\node[imgmut] (Q56)    at (7.2,-2.1)  {$\BQv{56}$};
\node[redmut] (P5789)  at (0,-4.2)    {$\RPl{5789}$};
\node[redmut] (R10589) at (3.6,-4.2)  {$\RRv{10}{589}$};
\node[imgmut] (Q59)    at (7.2,-4.2)  {$\BQv{59}$};
\node[imgfrz] (P6789)  at (0,-6.3)    {$\BPl{6789}$};
\node[imgmut] (R10789) at (3.6,-6.3)  {$\BRv{10}{789}$};
\node[imgmut] (Q89)    at (7.2,-6.3)  {$\BQv{89}$};
\node[frz] (P1234)   at (-3.4,-1.6)  {$\Pl{1234}$};
\node[mut] (P23410)  at (-3.4,-4.0)  {$\Pl{234\,10}$};
\node[frz] (P78910)  at (0,-8.6)     {$\Pl{789\,10}$};
\node[mut] (P18910)  at (3.6,-8.6)   {$\Pl{189\,10}$};
\node[frz] (P891011) at (0,-10.7)    {$\Pl{89\,10\,11}$};
\node[mut] (P151011) at (10.8,-5.4)  {$\Pl{15\,10\,11}$};
\node[frz] (P191011) at (10.8,-7.8)  {$\Pl{19\,10\,11}$};
\node[frz] (P2345)   at (14.4,-6.6)  {$\Pl{2345}$};
\node[mut] (R11256)  at (10.8,-1.0)  {$\Rv{11}{256}$};
\node[frz] (P12311)  at (10.8, 1.3)  {$\Pl{123\,11}$};
\node[frz] (P121011) at (10.8,-3.2)  {$\Pl{12\,10\,11}$};
\node[mut] (P2567)   at (14.4, 0.4)  {$\Pl{2567}$};
\node[mut] (P3567)   at (17.8, 0.4)  {$\Pl{3567}$};
\node[frz] (P4567)   at (21.2, 0.4)  {$\Pl{4567}$};
\node[mut] (P2356)   at (17.8,-2.0)  {$\Pl{2356}$};
\node[frz] (P3456)   at (21.2,-2.0)  {$\Pl{3456}$};
\begin{pgfonlayer}{background}
\draw[arr] (E)       -- (P5679);
\draw[arr] (P5679)   -- (R10567);
\draw[arr] (P5679)   -- (P5689);
\draw[arr] (R10567)  -- (R11567);
\draw[arr] (R10567)  -- (R10569);
\draw[arr] (Q56)     -- (R10567);
\draw[arr] (P5689)   -- (R10569);
\draw[arr] (P5689)   -- (P5789);
\draw[arr] (R10569)  -- (Q56);
\draw[arr] (R10569)  -- (R10589);
\draw[arr] (R10569)  -- (P5679);
\draw[arr] (Q59)     -- (R10569);
\draw[arr] (P5789)   -- (R10589);
\draw[arr] (P5789)   -- (P6789);
\draw[arr] (R10589)  -- (Q59);
\draw[arr] (R10589)  -- (R10789);
\draw[arr] (R10589)  -- (P5689);
\draw[arr] (Q89)     -- (R10589);
\draw[arr] (R10789)  -- (P5789);
\draw[arr] (Q56)     -- (Q59);
\draw[arr] (Q59)     -- (Q89);
\draw[arr] (P1234)   -- (R10567);
\draw[arr] (R11567)  to[bend left=8] (P1234);
\draw[arr] (P23410)  -- (P1234);
\draw[arr] (P23410)  -- (P78910);
\draw[arr] (R10789)  -- (P23410);
\draw[arr] (R10789)  -- (P18910);
\draw[arr] (P78910)  -- (R10789);
\draw[arr] (P18910)  -- (P78910);
\draw[arr] (P18910)  -- (Q89);
\draw[arr] (P891011) -- (P18910);
\draw[arr] (Q89)     -- (P191011);
\draw[arr] (P191011) -- (Q59);
\draw[arr] (Q59)     -- (P151011);
\draw[arr] (P151011) -- (P191011);
\draw[arr] (P151011) -- (P2345);
\draw[arr] (P2345)   to[bend left=8] (Q56);
\draw[arr] (Q56)     -- (R11256);
\draw[arr] (R11256)  -- (R11567);
\draw[arr] (R11256)  -- (P121011);
\draw[arr] (P121011) -- (Q56);
\draw[arr] (P12311)  -- (R11256);
\draw[arr] (R11567)  -- (P2567);
\draw[arr] (P2567)   -- (R11256);
\draw[arr] (R11256)  -- (P2356);
\draw[arr] (P2567)   -- (P3567);
\draw[arr] (P3567)   -- (P4567);
\draw[arr] (P3567)   -- (P2356);
\draw[arr] (P2356)   -- (P2567);
\draw[arr] (P2356)   -- (P3456);
\draw[arr] (P3456)   -- (P3567);
\draw[arr] (P2356)   to[bend right=32] (P12311);
\end{pgfonlayer}
\end{tikzpicture}
}
\end{center}
\caption{A seed $\overline{\Sigma}$ which contains the image of the seed from \cref{fig:rect2}. \label{fig:imagerect2}}
\end{figure}

 To obtain $\overline{\Sigma}$,  we
take the rectangles seed of $\Gr(4,11)$ in the cyclic order $5,6,\dots,11,1,2,3,4$, which has mutable vertices 
\begin{align*}
&v_{5679},v_{567\,10},v_{567\,11},v_{1567},v_{2567},v_{3567},
v_{5689},v_{569\,10},v_{56\,10\,11},v_{156\,11},v_{1256},v_{2356},
v_{5789},v_{589\,10},v_{59\,10\,11},\\&v_{15\,10\,11},v_{125\,11},v_{1235}.
\end{align*}
Now we apply the following sequence of 13 mutations.
\[
\mu_{v_{1235}},\mu_{v_{1256}},\mu_{v_{125\,11}},\mu_{v_{567\,11}},\mu_{v_{156\,11}},
\mu_{v_{1567}},\mu_{v_{567\,10}},\mu_{v_{56\,10\,11}},\mu_{v_{567\,11}},\mu_{v_{569\,10}},
\mu_{v_{59\,10\,11}},\mu_{v_{56\,10\,11}},\mu_{v_{589\,10}} .
\]

\begin{proof}[Proof of \cref{thm:6leg}]
For general $n$, we can similarly verify that $\Phi_n$ is a cluster quasihomomorphism.
To find a seed $\overline{\Sigma}$ containing the image under $\Phi_n$ of the rectangles seed, we do the following:
start from
the rectangles seed of $\Gr(4,n)$ in the cyclic order $5,6,\dots,n,1,2,3,4$, then apply the following sequence of 13 mutations:
\[
\begin{aligned}
&\mu_{v_{1235}},\ \mu_{v_{1256}},\ \mu_{v_{1,2,5,n}},\ \mu_{v_{5,6,7,n}},\ \mu_{v_{1,5,6,n}},\
\mu_{v_{1567}},\ \mu_{v_{5,6,7,n-1}},\\
&\mu_{v_{5,6,n-1,n}},\ \mu_{v_{5,6,7,n}},\ \mu_{v_{5,6,n-2,n-1}},\ \mu_{v_{5,n-2,n-1,n}},\
\mu_{v_{5,6,n-1,n}},\ \mu_{v_{5,n-3,n-2,n-1}} .
\end{aligned}
\]
The resulting cluster contains the image of the rectangles cluster under $\Phi_n$.  It looks like the seed in \cref{fig:imagerect2}, but each time we increase $n$ by $1$, we add a new column of mutable cluster variables at the left, which are all Pl\"ucker coordinates.
The seven extra mutable variables in $\overline{\Sigma}$ (beyond those that occur as images of mutable variables in $\Sigma$) are:
\[
\Pl{2356},\quad \Pl{2567},\quad \Pl{3567},\quad
\Pl{2,3,4,n-1},\quad \Pl{1,5,n-1,n},\quad \Pl{1,n-3,n-2,n-1},\quad
\Rv{n}{256},
\]
where $\Rv{n}{256}=\Pl{2,3,4,n}\Pl{1256}-\Pl{1234}\Pl{2,5,6,n}.$
As before, we can check that 
properties (1), (2), and (3) of \cref{def:quasi} are satisfied.    
\end{proof}

\subsection{Proof of the main theorem}
We are now ready to prove the main theorem of this section.  Our proof is analogous to the proof of 
\cref{thm:main1}.  

\begin{theorem}\label{thm:main2}
If a $\Gr(4,n)$ tensor diagram $D$ is a tree web, then 
$[D]$ is a cluster or coefficient variable for $\C[\widehat{\Gr}(4,n)]$.
\end{theorem}
\begin{proof}
    By \cref{prop:decompose}, each $\Gr(4,n)$ tree web can be built inductively using the four operations shown in \cref{fig:promotions_for_Gr4n}.  By \cref{lem:case1}, \cref{lem:case2}, \cref{lem:9leg}, and \cref{lem:6leg}, each of those operations corresponds to a cluster quasihomomorphism.  Since a cluster quasihomomorphism takes a 
    cluster or coefficient variable to a cluster or coefficient variable (possibly multiplied by a frozen factor), 
    we can inductively show that the invariant associated to a tree web is always a cluster or coefficient variable possibly multiplied by a frozen factor.

    But by the same arguments used in \cref{thm:main1}, we can show that the invariant associated to a tree web is irreducible.  In particular, the proof of \cite[Lemma 12.2]{FP16} applies in this situation.  Our alternative proof of irreducibility also applies, using the fact that the frozen variables that appear in 
    \cref{lem:case1}, \cref{lem:case2}, \cref{lem:9leg}, and \cref{lem:6leg} are either the cyclically consecutive Pl\"ucker coordinates, or (in the case of \cref{lem:9leg}) the quadratic cluster variable 
    ${\langle 12 |345 | 678 \rangle}$.  In all cases, these frozen variables are supported on consecutive boundary vertices, so the last argument of \cref{thm:main1} applies.
\end{proof}

\bibliographystyle{alpha}
\bibliography{biblio.bib}

\end{document}